\documentclass[10pt,reqno]{article}
\usepackage{graphicx}
\usepackage{amsmath,amsthm}
\usepackage{amsfonts}
\usepackage{amscd}
\usepackage[all]{xy}
\usepackage{amssymb}
\usepackage[a4paper]{geometry}
\usepackage{subfigure}
\usepackage{mathtools}
\usepackage{upgreek}
\usepackage{enumitem}
\usepackage{newtxtext}
\usepackage{newtxmath}
\usepackage{comment}
\usepackage{mathrsfs}
\usepackage{float}

\usepackage{color}
\usepackage{tikz}

\usepackage[colorlinks=true,linkcolor=blue,urlcolor=green]{hyperref}

\makeatletter
\def\referencia#1#2{\begingroup
    #2%
    \def\@currentlabel{#2}%
    \phantomsection\label{#1}\endgroup
}
\makeatother

\makeatletter
\let\@old@begintheorem=\@begintheorem
\def\@begintheorem#1#2[#3]{%
  \gdef\@thm@name{#3}%
  \@old@begintheorem{#1}{#2}[#3]%
}
\def\namedthmlabel#1{\begingroup
   \edef\@currentlabel{\@thm@name}%
   \label{#1}\endgroup
}
\makeatother

\newtheorem{theorem}{Theorem}[section]

\newtheorem{proposition}{Proposition}[section]
\newtheorem{lemma}{Lemma}[section]
\newtheorem{corollary}{Corollary}[section]
\newtheorem{definition}{Definition}[section]
\newtheorem{remark}{Remark}[section]

\newtheorem*{conjecture}{Conjecture}

\newtheorem*{theoremBKH}{Theorem: Partial flexibility}
\newtheorem*{theoremA}{Theorem A}
\newtheorem*{theoremB}{Theorem B}

\newcommand{\norm}[1]{\left\Vert#1\right\Vert}

\newcommand{\Real}{\mathbb R}
\newcommand{\Z}{\mathbb{Z}}
\newcommand{\Nat}{\mathbb{N}}

\newcommand{\BM}{\mathbf{B}(M)}

\newcommand{\Tor}{\mathbb{T}}

\newcommand{\oo}{\infty}

\newcommand{\p}{\mathtt{P}}
\newcommand{\mc}{\mu^{\xi}}
\newcommand{\PA}{\mathbb{P}(A)}

\newcommand{\disG}{\mathrm{dis}_{\mathrm{Gras}}}
\newcommand{\disH}{\mathrm{dis}_{\mathrm{Hff}}}
\newcommand{\disang}{d_{\angle}}
\newcommand{\lam}{\lambda}

\usepackage{authblk}

\begin{document}

  \title{Extended flexibility of Lyapunov exponents for Anosov diffeomorphisms}
  \author[1]{Pablo D. Carrasco \thanks{pdcarrasco@mat.ufmg.br}}
  \author[2]{Radu Saghin \thanks{radu.saghin@pucv.cl Supported by Fondecyt Regular 1171477, 1210168}}
  \affil[1]{ICEx-UFMG, Avda. Presidente Antonio Carlos 6627, Belo Horizonte-MG, BR31270-90}
  \affil[2]{Pontificia Universidad Cat\'olica de Valpara\'{\i}so, Blanco Viel 596,
  Cerro Bar\'on, Valpara\'{\i}so-Chile.}

  \date{\today}

  \maketitle

\begin{abstract}
Bochi-Katok-Rodriguez Hertz proposed in \cite{Bochi2019} a program on the flexibility of Lyapunov exponents for conservative Anosov diffeomorphisms, and obtained partial results in this direction. For conservative Anosov diffeomorphisms with strong hyperbolic properties we establish extended flexibility results for their Lyapunov exponents. We give examples of Anosov diffeomorphisms with the strong unstable exponent larger than the strong unstable exponent of the linear part. We also give examples of derived from Anosov diffeomorphisms with the metric entropy entropy larger than the entropy of the linear part.

These results rely on a new type of deformation which goes beyond the previous Shub-Wilkinson and Baraviera-Bonatti techniques for conservative systems having some invariant directions. In order to estimate the Lyapunov exponents even after breaking the invariant bundles, we obtain an abstract result (Theorem \ref{teo.tecnico}) which gives bounds on exponents of some specific cocycles and which can be applied in various other settings. We also include various interesting comments in the appendices: our examples are $\mathcal{C}^2$ robust, the Lyapunov exponents are continuous (with respect to the map) even after breaking the invariant bundles, a similar construction can be obtained for the case of multiple eigenvalues.
\end{abstract}


\section{Introduction}\label{sec.introduction}

One of the central tasks in Global Analysis is to try to understand the structure of the diffeomorphism group of a given manifold. For this, and in virtue of the difficulty of the problem, it is usually imposed the existence of some invariant geometrical structure, and the focus is on the subgroup of diffeomorphism preserving such structure. This is the point of view taking here: we consider a closed Riemannian manifold $M$ and restrict our considerations to $\mathrm{diff}^r_{\mu}(M)$, the subgroup of $\mathcal{C}^r$ smooth diffeomorphism preserving some smooth volume $\mu$, which, for definitiveness, can be taken as the volume element induced by the Riemannian metric.

For $f\in \mathrm{diff}^r_{\mu}(M)$ the Oseledec's theorem together with Pesin's theory provide the existence of  invariant manifolds, defined almost everywhere in the manifold and depending measurably on the base point, such that the asymptotic growth of the tangent vectors to such manifolds are well defined. These are the the \emph{Lyapunov exponents}, and since our article deals with these quantities, we recall the precise definition below.

\begin{definition}
For a $\mathcal{C}^r$ diffeomorphism $f:M\to M, x\in M$, $r\geq 1$ and $x\in M, v\in T_xM\setminus\{0\}$ the Lyapunov exponent (of $f$) at $(x,v)$ is the quantity
\[
	\bar{\chi}(x,v)=\limsup_{n\to\oo} \frac{1}{n}\log \|D_xf^n(v)\|.
\]
If the above limit exists we denote by $\chi(x,v)$ its value.  
\end{definition}

If $f\in\mathrm{diff}^r_{\mu}(M), m=\dim M$ there exists $M_0 \subset M$ of full measure and measurable functions $\chi^1,\cdots, \chi^m:M_0\to\Real, d:M_0\to \{1,\cdots, m\}$ such that for $x\in M_0$ it holds:
\begin{itemize}
	\item $\chi^1(x)>\chi^2(x)>\cdots>\chi^{d(x)}(x)$; this is called the Lyapunov spectrum of the system.
	\item $E^i(x):=\{v\in T_x M: v=0\text{ or }\chi(x,v)=\chi^i(x)\}$ is a subspace of $T_xM$ for $1\leq i\leq d(x)$ and $T_xM=\bigoplus_{i=1}^{d(x)}E^i(x)$. Moreover, $x\mapsto E^i(x)$ is measurable (as map from $M_0$ to the Grassmanian of $M$).
\end{itemize}
As the functions $d,\chi^i, \dim E^i$ are $f$ invariant, ergodicity of the system $(f,\mu)$ implies that these functions are constant almost everywhere. Assuming this to be the case (which is not a serious restriction as one can always decompose $\mu$ into ergodic components), and under a mild extra regularity assumption of $f$ one obtains furthermore that $E^u:=\bigoplus_{\chi^i>0} E^i, E^s:=\bigoplus_{\chi^i<0} E^i$ are integrable to laminations with smooth leaves. These laminations then provide one of the referred invariant structures that may aid for the understanding of the system. 

The dimension of the subspace $E^i(x)$ is called the multiplicity of the exponent $\chi^i(x)$, and the sum of all the multiplicities is $m$, the dimension of manifold. It is useful sometimes to consider an alternative notation of the Lyapunov exponents, each one of them counted with its multiplicity, and thus we have $m$ exponents for almost every point, possibly repeating themselves if $d(x)<m$:
$$
\chi^1(x)\geq\chi^2(x)\geq\cdots\geq\chi^m(x).
$$

It is natural then to inquiry whether there exists any limitation for these invariant data. Concretely, one can ask the following.

\noindent\textbf{Question:} if $g:M\to M$ is a diffeomorphism preserving $\mu$ isotopic/homotopic to $f$, what are the restrictions on the Lyapunov spectrum of $g$?  

\smallskip 

Lack of such restrictions is known as flexibility, and it has been receiving a quite deal of attention lately, impulsed by the program proposed by J. Bochi, A. Katok and F. Rodriguez-Hertz (\cite{Bochi2019}; see also \cite{E2019,EK2019,BE2020} for further results on flexibility). The difficulty in dealing with this type of question resides in that there are only a few methods known to perturb a system inside $\mathrm{diff}^r_{\mu}(M)$ that permit to control the behavior of the Lyapunov exponents.

\paragraph{Perturbation methods for conservative systems} Probably the first of such constructions is the one appearing in the classical Shub-Wilkinson example \cite{Patho}, which led to the local version \cite{BARAVIERA2003}. The argument relies, roughly speaking, in mixing two invariant directions, and the consequence is that the two corresponding Lyapunov exponents will move closer to each other. In \cite{Bochi2019} the authors extend this method to a semi-local type argument, by making several local perturbations in different parts of the manifold. Controlling the interactions between these and obtaining precise control for the exponents is delicate. 

However, these methods seem to have some restrictions. Since they depend on the persistence of the aforementioned invariant directions, one can only move pairs of Lyapunov exponents closer together, and not further apart. Also the sum of the two exponents stays constant, and this implies that the metric entropy can only be decreased.

In this article we present some global perturbation methods for conservative diffeomorphisms that will allow us to address these type of difficulties, and apply them to the flexibility program. Nonetheless, due to their generality we expect them to be applicable in other contexts. We consider larger perturbations mixing strongly two invariant directions, such that the splitting is destroyed (the two dimensional bundle remains however invariant). We are able then to obtain bounds on the new Lyapunov exponents which will show us that the exponents can be moved further apart, and even can change the sign, and thus increase the metric entropy.

\subsection{Flexibility for Anosov diffeomorphisms} To have some hope of answering the flexibility question, some restrictions are needed. A natural condition consists of starting from a map $f$ whose dynamics is well understood and studying what happens under deformations. This is the approach taken by Bochi et al., and the one that we follow here. Concretely, the authors restrict their attention to the set $\mathscr{A}$ of Anosov diffeomorphisms of a manifold $M$, which for definiteness they take $M=\Tor^d=\Real^d/\Z^d$. Denoting by $\mu$ the corresponding Haar (probability) measure, we let $\mathscr{A}_\mu=\mathrm{diff}^r_{\mu}(\Tor^d)\cap \mathscr{A}$.

In this case, any two homotopic elements in $\mathscr{A}$ are topologically conjugate \cite{Franks1969,Manning1974}, and in particular have the same topological entropy. Moreover, any $f\in\mathscr{A}$ is homotopic, and therefore topologically conjugate to the linear automorphism $L$ induced by its action on $H_1(\Tor^d)\approx \Z^d$; $L$ is called the linear part of $f$. For such an automorphism, its Lyapunov exponents $\log|\lambda^i|$ are given by its eigenvalues $\lambda^i$. Let us also point out that for $f\in \mathscr{A}_{\mu}$ we have the following restrictions on its spectrum:
\begin{align*}
\MoveEqLeft \int \sum_{i=1}^d \chi^i_f\cdot d\mu =0\\
h_{\mu}(f)&= \sum_{i=1}^d \int\chi^i_f\vee 0\cdot d\mu\leq h_{top}(f)=h_{top}(L)= \sum_{i=1}^d \int \log|\lambda^i|\vee 0\cdot d\mu.
\end{align*}
Above $h_{\mu}(f)$ denotes the metric entropy of $f$; this second part is a classical result due to Pesin, and we direct the reader to \cite{Barreira2013} for a proof of both these properties.

The questions of interest are the following.

\smallskip 
 
\noindent\textbf{Question 1: Strong flexibility - Anosov} (Problem 1.3 in \cite{Bochi2019}) Given $L\in\mathrm{Sl}(d,\Z)$ a linear automorphism of $\Tor^d$ with eigenvalues $|\lam^1|\geq |\lam^2|\geq\cdots |\lam^u|>1>|\lam^{u+1}|\geq\cdots |\lam^d|$, and given numbers $\eta^1\geq \eta^2\geq \cdots \eta^u>0>\eta^{u+1}\geq\cdots \eta^d$ satisfying
\begin{enumerate}
	\item[a)] $\sum_{i=1}^d\eta^i=0$;
	\item[b)] $\sum_{i=1}^u\eta^i\leq \sum_{i=1}^u\log|\lam^i|$.
\end{enumerate}
Does there exist $f\in\mathscr{A}_{\mu}$ homotopic to $L$ such that $\chi^i_f=\eta^i$ for every $i$? 

\smallskip 

The above implies in particular that for any $0<h\leq \sum_{i=1}^u\lam^i$ there exists $f\in \mathscr{A}_{\mu}$ homotopic to $L$, with the same unstable dimension as $L$ and with $h_{\mu}(f)=h$.

\smallskip 

\noindent\textbf{Question 2: Strong flexibility - general} (Conjecture 1.2 in \cite{Bochi2019} for the homotopy class of $L$) Given $L\in\mathrm{Sl}(d,\Z)$ a linear automorphism of $\Tor^d$, and given non-zero numbers $\eta^1\geq \eta^2\geq \cdots \eta^d$ satisfying
\begin{enumerate}
	\item[a)] $\sum_{i=1}^d\eta^i=0$;
\end{enumerate}
Does there exist a conservative (ergodic) $f$ homotopic to $L$ such that $\chi^i_f=\eta^i$ for every $i$? 

\smallskip 
 
As no restrictions in the topological entropy are imposed, the map $f$ may not be Anosov.

The Question 2 seems difficult to approach at this moment in its generality. In the absence of some hyperbolicity it is arduous to conclude the ergodicity of a map, let alone estimating the Lyapunov exponents. So it seems natural to ask an intermediate question within the partially hyperbolic setting, where one can use some geometric objects in order to both prove ergodicity and obtain control on the expansion rates.

\smallskip 
 
\noindent\textbf{Question 3: Strong flexibility - partially hyperbolic} Given $L\in\mathrm{Sl}(d,\Z)$ a linear automorphism of $\Tor^d$ with eigenvalues $|\lam^1|\geq|\lam^2|\geq\cdots |\lam^u|>1>|\lam^{u+1}|\geq\cdots |\lam^d|$, and given non-zero numbers $\eta^1\geq \eta^2\geq \cdots \eta^d$ satisfying
\begin{enumerate}
	\item[a)] $\sum_{i=1}^d\eta^i=0$;
	\item[b)] $\eta^1\leq\log|\lam^1|$;
	\item[c)] $\eta^d\geq\log|\lam^d|$.
\end{enumerate}
Does there exist a conservative (ergodic) partially hyperbolic $f$ map homotopic to $L$ such that $\chi^i_f=\eta^i$ for every $i$? 

\smallskip 

Since in this case the exponents can both decrease and increase, and even change sign, it is possible that the metric and topological entropy of $f$ are greater than the entropy of $L$.

Question 3 corresponds to partially hyperbolic maps with one dimensional stable and unstable bundles, but one can also consider corresponding questions for higher dimensional bundles, or within the weak partially hyperbolic setting.
 
In \cite{Bochi2019} the authors are able to give a partial answer to Question 1 assuming simplicity of the spectrum of $L$.

\begin{theoremBKH} (Theorem 1.5 in \cite{Bochi2019}) Let $L\in\mathrm{Sl}(d,\Z)$ be a linear automorphism of $\Tor^d$ with eigenvalues $|\lam^1|> |\lam^2|>\cdots |\lam^u|>1>|\lam^{u+1}|>\cdots |\lam^d|$, and given non-zero numbers $\eta^1> \eta^2> \cdots\eta^u>0>\eta^{u+1}>\cdots \eta^d$ satisfying
\begin{enumerate}
	\item[a)] $\sum_{i=1}^d\eta^i=0$;
	\item[b)] $\sum_{i=1}^k\eta^i\leq \sum_{i=1}^k\log|\lam^i|,\ \ \forall k\in\{1,2,\dots d-1\}.$
\end{enumerate}
Then there exists $f\in \mathscr{A}_{\mu}$ such that $\chi^i_f=\eta^i$ for every $i$.
\end{theoremBKH}

However, their methods do not seem to be fitted to give a complete answer to Question 1, and do not seem to work well for Questions 2 or 3 (although they may give a partial answer for Question 3). In particular, as we mentioned, they are not able to deal with the problem of raising strong unstable exponents, or the sum of the unstable exponents. In consequence, they asked whether the following is true.

\medskip 
 
\noindent\textbf{Question:}
{\it 
Does there exist $f:\Tor^3\to \Tor^3, f\in \mathscr{A}_{\mu}$ with linear part $L$ such that the largest Lyapunov exponent of $f$ is larger than the logarithm of the largest eigenvalue of $L$?
}

\medskip

For this to have a positive answer, necessarily the dimension of the unstable bundle of $f$ has to be two. To our knowledge there is no such example in the literature.

Our first main result gives an affirmative answer to this question.

\begin{theoremA}\hypertarget{theoremA}{}
Let $L$ be a hyperbolic automorphism of $\mathbb T^3$ with the eigenvalues $\lambda_{su}>\lambda_{wu}>1>\lambda_s$. Then there exist $n\in\Nat$ and a $C^{\oo}$ diffeomorphism $f$, isotopic to the identity and preserving the volume, such that $f\circ L^n$ is Anosov, preserves the volume, and the highest Lyapunov exponent of $f\circ L^n$ is strictly greater than $\log\lambda_{su}^n$.  
\end{theoremA}

\begin{remark}
Higher dimensional constructions can also be done by the same methods. Since the adaptations are straightforward and writing the general theorem would only complicate the statement, we decide to leave this to the readers. Similar considerations apply for \hyperlink{theoremB}{Theorem B}.
\end{remark}

The iterate $L^n$ is performed to guarantee strong hyperbolic conditions (i.e.\@ being far away from the boundary of $\mathscr{A}_{\mu}$). One can ask if this is necessary, and whether it is possible to prove the same result for $L$ instead of an iterate. At this moment this question remains unknown, however, in contrast to the case when the exponents are lowered, it seems unlikely that one can raise the largest exponent in the same homotopy class without any restriction. For a generic small perturbation of $L$ the entropy maximizing measure will be in general different from the volume, therefore by the variational principle both the highest exponent and the sum of the two positive exponents has to decrease; this means that the perturbation needed in order to make the highest exponent increase again has to be considerably large. If $L$ is close to the boundary of $\mathscr{A}_{\mu}$ it is possible that such a large perturbation necessarily takes the map outside the Anosov realm. Although in principle we do not show that this is an obstruction for the final map to be Anosov, we want to point out to the reader the possible existence of additional restrictions to achieve full flexibility. We also have some doubts whether it is possible to make within $\mathscr{A}_{\mu}$ the highest exponent arbitrarily close to the sum of the positive exponents of the linear part (this would imply that the weak unstable exponent is arbitrarily close to zero). With our method we can increase the highest exponent only with at most two thirds of the weak unstable exponent.

As pointed out in \cite{Bochi2019}, for contructing the examples of \hyperlink{theoremA}{Theorem A} necessarily one has to break the domination inside the unstable bundle. If this were not the case, the strong unstable foliation 
of $f\circ L^n$ would be quasi-isometric and therefore by \cite{DynCoh3Torus} it would have the same growth rate as $L^n$ (and therefore the same exponent).

\smallskip

Continuing with Question 3, there is an example by Ponce-Tahzibi (\cite{PT}), where they succeed in changing the sign of the middle exponent by perturbing some specific Anosov automorphisms on $\mathbb T^3$, within the space of derived from Anosov partially hyperbolic diffeomorphisms. They do it by mixing the weak stable direction with the unstable direction, so in their construction the sum of the positive exponents actually drops. An interesting question is the following.

\medskip

\noindent\textbf{Question:}
{\it 
Does there exist a derived from Anosov diffeomorphism $f$ such that the sum of the positive Lyapunov exponents of $f$ is larger than the sum of the positive Lyapunov exponents of the linear part $L$?
}

\medskip

Let us remark that this amounts to raise the metric entropy (so the topological entropy also increases). Of course this cannot be achieved inside $\mathscr{A}_{\mu}$, nonetheless one can do this within the partially hyperbolic setting.

\begin{theoremB}\hypertarget{theoremB}{}
Let $L$ be a hyperbolic automorphism of $\Tor^4$ with the eigenvalues $\lambda_{u}>1>\lambda_{ws}>\lambda_{ms}>\lambda_{ss}$. Then there exist $n\in\Nat$ and a $C^{\oo}$ diffeomorphism $f$, isotopic to the identity and preserving the volume, such that $L^n\circ f$ is weakly partially hyperbolic, preserves the volume, and almost every point ($L^n\circ f$ may not be ergodic) has two strictly positive Lyapunov exponents, one of which is the unstable exponent $\log\lambda_{u}^n$ (the positive exponent of $L^n$). In particular the metric entropy with respect to the volume of $L^n\circ f$ is strictly greater than the metric entropy of $L^n$.

If furthermore $\lambda_{ws}\lambda_{ms}>\lambda_{ss}$, then $L^n\circ f$ can be chosen such that it is strongly partially hyperbolic.
\end{theoremB}

This construction exhibits a remarkable feature: by mixing enough two stable directions one can produce an additional positive exponent. Since the perturbation needed here is fairly large, we need at least four dimensions for this construction. We expect that the construction can be made ergodic if we assume sufficient bunching ($\lambda_{ss}$ is large enough compared to $\lambda_{ms}$). Some arguments in favor are the stability of the examples explained in Appendix B, the classical results on the ergodicity of partially hyperbolic diffeomorphisms \cite{ErgPH}, and results on the density of accessibility like \cite{HS}, \cite{AV2020}; we point out however that it would be necessary to establish accessibility by $\mathcal{C}^2$ perturbations, which is, at the moment of writing, not completely known in full generality. An alternative approach to establish ergodicity, or at least to construct ergodic examples could be the methods used by Obata in \cite{Obata}. \label{ergodicidad}

\subsection{Ideas of the proofs}

The basic idea is to mix enough two stable directions of the Anosov automorphism such that the strong stable exponent actually decreases, while the weak stable exponent increases (this means that we work with $L^{-1}$ for Theorem A). The mixing is done by a one-parameter family of smooth diffeomorphisms $f_t$ which preserve the stable planes, and introduces a strong shear within these planes. This shear will create some hyperbolicity proportional with the parameter $t$ in a large ``good region''. 

This means that we gain some expansion in this good region, however there is a difficulty in the fact that the good region is not invariant, and once an orbit goes out of the good region we may loose control on the expansion which we gained. In order to deal with this problem we use ideas from \cite{NUHD}, \cite{RandomProdSt}. We consider a family of unit vector fields inside $E^s$ which are Lipschitz along $W^u$. For proper parameters $t,n$, this family is invariant under the push forward of $L^n\circ f_t$. If such a vector field is inside the expanding cone, then it gets expanded inside the good region, and most of it stays inside the expanding cone. If the vector field is not in the expanding cone, then we are able to show that a good portion of it has to go in the expanding cone after one iteration. Using these facts, some combinatorial arguments allow us to give a lower bound for the higher Lyapunov exponent inside $E^s$. The abstract argument to guarantee positivity of the exponents is isolated in Theorem \ref{teo.tecnico} with the intention that it can be used in other cases.

\subsection*{Organization of the article} In the next section we introduce the definitions and cover some preliminaries needed for the proofs. In the third section it is presented an abstract theorem that will be used later to establish the variation of Lyapunov exponents. The section four is devoted to describing the family of perturbations $f_t$ and their various properties needed for the proofs of Theorems A and B. Finally Theorem A is proved in Section 5 and Theorem B is proved in section 6. There are also several appendices where some complementary aspects are discussed; these are given to mainly to encourage the reader to investigate other properties of the examples constructed.  In Appendix A some additional considerations related to the abstract theorem are given. In Appendix B we show that the results are robust with respect to $\mathcal{C}^2$ perturbations. In Appendix C we discuss the continuity of the exponents with respect to the perturbation, while in Appendix D we also show how to increase the top exponent in the case of higher multiplicity, using the Invariance Principle \cite{AV2010}. Finally, Appendix E frames the examples of \hyperlink{theoremB}{Theorem B} into the theory of physical measures.

\subsection*{Acknowledgements} The authors thank the referee for her/his careful reading, and for various corrections and suggestions. The first author thanks the hospitality of the PUCV math department where this project was started.

\section{Preliminaries}\label{sec.preeliminaries}

In this section we collect some notations and give some references for results that will be needed. Although some of these results are valid in more generality, we will content ourselves with stating them in the context that they will be used in the article.

\subsection{Measurable partitions and conditional measures}

Let $(M,d_M)$ be a compact metric space, and denote $\BM$ its Borel $\sigma$-algebra. By a \emph{measure} on $M$ we mean a Borel probability measure. Fixing such a measure $\mu$, we say that $\xi$ is a partition of $M$ (with respect to $\mu$) if $\xi\subset\BM$ and 
\[
	\p,\p'\in \xi, \p\neq \p'\Rightarrow \mu(\p\cap \p')=0.
\]
Elements of a partition are called \emph{atoms}, and for a given partition $\xi$ and $\mu$-almost every $x\in M$, the unique atom containing $x$ is denoted by $\xi(x)$. Given $\xi_1,\xi_2$ partitions we say that $\xi_2$ is finer than $\xi_1$ if every atom of $\xi_1$ is union of atoms of $\xi_2$: in this case we write $\xi_1\leq \xi_2$. For $\xi_1,\xi_2,\ldots,\xi_k$ partitions, we write 
\[
	\bigvee_{i=0}^k \xi_i =\{\p_1\cap\p_2\cap\cdots\p_k: \p_i\in \xi_i\}.
\]
Clearly $\bigvee_{i=0}^k \xi_i$ is a partition (by measurable sets) that is finer than each $\xi_i$.

\begin{definition}
A partition $\xi$ of $M$ is called measurable if there exists a sequence $\xi_1\leq \xi_2\leq\cdots\xi_n\cdots$ of finite partitions satisfying: there exists $M_0\subset M$ of full $\mu$-measure such that for every $x\in M_0$,
\[
	\xi(x)=\lim_{n\to\oo} \xi_n(x).
\]
\end{definition}

Fix $\xi$ a measurable partition. The space of atoms of $\xi$ is denoted $M/\xi$, and it is called the \emph{factor space} of $M$ with respect to $\xi$; associated to this space one can define a \emph{factor map} $\pi:M\to M/\xi$ by $\pi(x)=\xi(x)$ for $x\in M_0$, and $\pi(x)=\xi(x_0)$ for $x\in M\setminus M_0$, where $x_0$ is some arbitrary point in M (of course, $\pi$ is not uniquely defined). The space $M/\xi$ is equipped with the largest $\sigma$-algebra making $\pi$ measurable, and we denote by $\mu_{M/\xi}=\pi_{\ast}\mu$ the induced \emph{factor measure}. For different choices of $M_0, \pi$ the resulting factor spaces are isomorphic in the sense of measure theory; we fix (arbitrarily) one of these versions to work.

It is then a result of V. Rokhlin \cite{Rokhlin} that there exist a (essentially unique) family of probability measures $\{\mc_{\p}\}_{\p \in M/\xi}$ such that for every $A\in\BM$,
\[
	\mu(A)=\int_{M/\xi} \mc_{\p}(A\cap \p)d\mu_{M/\xi}(\p), 
\]
where $\mc_{\p}$ is supported on the atom $\p$. The family $\{\mc_{\p}\}_{\p \in M/\xi}$ is said to furnish a \emph{disintegration} of $\mu$. Note that if $\xi$ is finite then $\mc_{\p}$ is just the conditional measure on $\p$. 

As a convention, if the support atom of $\mc_{\p}$ is clear from the context we omit its reference and we simply write $\mc$. 

\subsection{Cocycles and Lyapunov exponents}\label{subsec.cocyclesandlyapunov}

Let $f:M\to M$ be a continuous endomorphism of the compact metric space $M$, and suppose that $\pi:E\to M$ is a (continuous) finite rank vector bundle over $M$. For what follows it will be convenient to assume that $E$ is a sub-bundle of $\mathbb{R}^N$, for some $N$; this is no loss in generality since $E$ is a direct summand of a trivial vector bundle.  A (continuous) bundle map $A:E\to E$ is called a cocycle over $f$ provided that its base map coincides with $f$. In this case we write for $x\in M$, $A(x):E_{x}=\pi^{-1}(x)\to E_{fx}$ the corresponding linear isomorphism. We will adhere to common practice of referring to the pair $(f,A)$ as the cocycle; if needed we can consider $A(x)$ as an $N\times N$ square matrix of size equal to the rank of $E$. 

The vector bundles that we will consider have some additional structure; we will assume from now on that
\begin{itemize}
\item $E$ is $\theta$-H\"older for some $0<\theta\leq 1$ (when considered as a section of the corresponding Grassmanian of $\mathbb{R}^N$).
\item $E$ is equipped with the restriction of some Riemannian metric in $\mathbb{R}^N$
\end{itemize}

Let us say a few words about the H\"older condition. For $m\in M$ let $\mathrm{proj}_m:\mathbb{R}^N\to E_m$ be the orthogonal projection. The distance between $E_m$ and $E_m'$ is
\[
	\disG(E_m,E_{m'})=\max\left\{\norm{v-\mathrm{proj}_m(v)}, \norm{v'-\mathrm{proj}_{m'}(v')}:v\in E_m,v'\in E_{m'}, \norm{v}=\norm{v'}=1\right\}
\]
To say that $E$ is $\theta$-H\"older means that there exists some $C_E>0$ such that for every $m,m' \in M$,
\[
	\disG(E_m,E_{m'})\leq C_E\cdot d_M(m,m')^{\theta}.
\]
The total variation of $E$ is the diameter of the image of the induced section in the Grassmanian, 
\[
	\mathrm{var}(E)=\sup_{m,m'}\{\disG(E_m,E_m')\}.
\]
In this article we will only deal with bundles with small variation (arising as perturbations of constant bundles); it will be helpful for the reader to keep this in mind.

\smallskip

If $(f,A)$ is a cocycle we write $\norm{A}=\max_{x\in M}\{\norm{A(x)}\}$, where $\norm{A(x)}$ is the operator norm the linear map with respect to the metrics on $E_x,E_{fx}$.

For a cocycle $(f,A)$ and $n \in \Z$ we denote
\[
	A^{(n)}(x)=\begin{cases}	
	A(f^nx)\cdot A(f^{n-1}x)\cdots A(fx)\cdot A(x) & n\geq 0\\
	Id & n=0\\
	A^{-1}(f^{-1}x)\cdot A^{-1}(f^{-2}x)\cdots A^{-1}(f^{-n}x) & n<0\quad (\text{ provided that }f^{-1}\text{ exists}).
	\end{cases}	
\]

Let us now assume that the rank of $E$ is $2$ and that $\mu$ is an ergodic invariant measure for $f$; it is a result originally due to Furstenberg and Kesten \cite{FurstKest60} that in this case, for almost every $x\in M$ the following limits exist and do not depend on $x$:
\begin{align*}
&\chi^+=\lim_{n\to +\oo}\frac{\log\norm{A^{(n)}(x)}}{n}\\
&\chi^{-}=-\lim_{n\to +\oo}\frac{\log\norm{(A^{(n)}(x))^{-1}}}{n}.
\end{align*}
One verifies directly that the limits do not depend on the Riemannian metric chosen and $\chi^+\geq \chi^-$: these numbers are called respectively the largest and smallest Lyapunov exponents. Even more, Oseledet's theorem \cite{Oseledets} guarantees that for $\mu$ almost every $x$ and every $v\in E_x\setminus\{0\}$, $\chi(x,v)=\lim_{n\to +\oo}\frac{\log\norm{A^{(n)}(x)}}{n}$ exists and coincides with either $\chi^+$ or $\chi^{-1}$. In the case when $\chi^+\neq \chi^{-}$ one can find measurable line bundles $E^{+},E^{-}$ such that $\mu$-almost every $x$, $E_x=E^+_x\oplus E_x^{-}$, and for $v\in E^{\sigma}\setminus\{0\}, \chi(x,v)=\chi^{\sigma}, \sigma=+,-$. If $\mu$ is not ergodic then $\chi^+(x)$ and $\chi^-(x)$ are functions depending on the point $x$, however they are well defined for almost every $x$, measurable and invariant.

\subsection{Projective cocycles}
For a (finite dimensional) vector space $V$ we denote $\mathbb{P}V$ its corresponding projective space, and likewise if $\pi:E\to M$ is a (finite rank) vector bundle we denote $\mathbb{P}E=\bigsqcup_{m\in M} \mathbb{P}E_m$ the projective bundle over $M$ induced by $E$. In the case when $V$ is two dimensional $\mathbb{P}V$ can be identified by a circle: for this we assume that $V$ is equipped with an inner product and for a chosen unit length $e_1\in V$ we identify 
\[
	\mathbb{P}V\approx\{v\in  V:\norm{v}=1,\angle(v,e_1)\in [0,\pi)\}\approx \mathbb{S}^1\subset \Real^2.
\]
Note that this identification depends on $e_1$ and we equip $\mathbb{P}V$ with the distance 
\[
	\disang([v],[v'])=|\angle(v,v')|
\]
where the angle is measured with the inner product in $V$. Using this, in practice we write $v\in \mathbb{P}V$ instead of 
$[v]\in \mathbb{P}V$ (provided that $v\neq 0$, of course).  A similar discussion can be applied to the case when $\pi:E\to M$ is a rank two vector bundle.

\begin{definition}		
A subset $\Delta\subset \mathbb{P}V$ will be called a sector if it is union of finitely many intervals.
\end{definition}

Notions as \emph{closed sector}, \emph{open sector} and \emph{boundary} of the sector are self-explanatory. We note in particular that if $\Delta$ is a symmetric cone in $V$ then it defines naturally a sector in $\mathbb{P}V$, and vice-versa. If $\Delta$ is a sector its complementary sector is $\Delta^c=\mathbb{P}V\setminus\Delta$.

To compare sectors in the same projective space we use the Hausdorff distance between their closures. Observe also that if $\Delta\subset \mathbb{P}V$ is a sector and $\mathbb{P}L:\mathbb{P}V\to \mathbb{P}V'$ is induced by a linear isomorphism $L:V\to V'$, then $L^{\ast}\Delta=PL(\Delta)$ is a sector in $PV'$.

The notion of sector extends naturally to the case when $\pi:E\to M$ is a rank two vector bundle: in this case a sector is a collection $\Delta=\{\Delta_m\}_{m\in M}$ where $\Delta_m\subset \mathbb{P}E_m$ is a sector for every $m$, that depends continuously on $m$. For comparing sectors in different fibers we can use orthogonal projections as we did before, and in particular define the total variation of a sector $\Delta$ in $\mathbb{P} E$ by
\[
	\mathrm{var}(\Delta)=\sup_{m,m'\in M}\{\disH(\Delta_m,\mathrm{proj}_m^{\ast}\Delta_m')\}.
\]

\smallskip

Now given a cocycle $(f,A)$ acting on $\pi:E\to M$ we consider $(f,\PA):\mathbb{P}E=M\times \mathbb{P}\Real^1\to \mathbb{P}E$ its projectivization:
\[
	v\in \mathbb{P}E_m\Rightarrow \PA(m)\cdot v=[A(m)\cdot v]=\pm \frac{A(m)\cdot v}{\norm{A(m)\cdot v}}.
\]

We remark that for every $m\in M$ the map $\PA(m):\mathbb{P}E_m\to \mathbb{P}E_{fm}$ is Lipschitz with Lipschitz constant $\norm{A(m)}\cdot \norm{A(fm)^{-1}}$.

\begin{definition}
If $(f,A)$ is a cocycle we define
\[
	b(A)=\sup_{m\in M} \norm{A(m)}\cdot \norm{A(fm)^{-1}}
\]
\end{definition}

We finish this part noting that if $(f,A)$ is a H\"older cocycle then $(f,\PA)$ is H\"older as well.

\subsection{Vector fields on increasing partitions} We keep the previous hypotheses on $f, A$ and $\pi:E\to M$. 

We will say that a measurable partition $\xi$ of $M$ is increasing for $f$ if it satisfies for every $\p_0\in \xi$:
\begin{itemize}
	\item $f|\p_0$ is one-to-one and positive measurable;
	\item $f(\p_0)$ is a finite union of atoms of $\xi$.
\end{itemize}
We denote $\xi^1=\xi\vee f^{-1}\xi$, and for $\p\in\xi^1$ we write $g_{\p}:f(\p)\to \p$ the corresponding inverse branch of $f$.   

Fix $\xi$ increasing for $f$, and let $\p_0\in\xi$. Given $X:\p\to E$  unit length vector field, it induces a section in $\mathbb{P}E$ which we can safely denote by the same letter. Observe that for every $\p\in \xi^1|\p_0$ we can use $f$ to push forward the vector field and define $Y^1_{\p}:f(\p)\to E$ by
\[
Y^1_{\p}(m):=\PA(g_{\p}m))\cdot  X(g_{\p}m)\quad m\in f(\p)
\] 
Analogously, for a natural number $k\geq 1$ we denote $\xi^k=\bigvee_{i=0}^{k} f^{-i}\xi$. Note in this case that for every atom $\p\in \xi^k$ there exists well defined inverse of $f^k:\p\to f^k(\p) (f^k(\p)\in \xi)$, that we denote $g_{\p}^k$. We can then proceed analogously as in the case $k=1$ and define for every $\p\in\xi^k|\p_0$ a section $Y^k_{\p}:f^k(\p)\to \mathbb{P}E$ by the formula
\[
Y^k_{\p}(m):=\PA^{(k)}(g^k_{\p}(m))\cdot X(g_{\p}^km).
\]

\begin{definition}
Let $\xi$ be an increasing measurable partition for $f$. A family $\mathcal{X}$ of vector fields  over atoms of $\xi$ is said to be invariant under the cocycle $(f,A)$ if $X\in \mathcal{X}$ implies $Y^k\in \mathcal{X}$ for every $k\geq 1$.
\end{definition}


\subsection{Partial Hyperbolicity}

We suppose that $M$ is a compact (boundaryless) manifold and $f:M\to M$ is a diffeomorphism os class at least $\mathcal{C}^1$. 

\begin{definition}
$f$ is weakly partially hyperbolic if there exist a $Df$ invariant splitting $TM=E^{cs}\oplus E^u$, $\lambda>1$
and a Riemannian metric on $M$ such that for every $m\in M$, for every unit vectors $v\in E^u_m,w\in E_m^{cs}$ it holds
\begin{enumerate}
\item $\norm{D_mf(v)}\geq \lambda$\  (uniform expansion in $E^u$);
\item $\norm{D_mf^n(w)}<\frac{1}{2}\cdot \norm{D_mf^n(v)}$\   (domination between $E^{cs}$ and $E^u$).
\end{enumerate}
$f$ is partially hyperbolic if both $f$ and $f^{-1}$ are weakly partially hyperbolic.
\end{definition}

The bundles $E^{cs},E^u$  are the center-stable and unstable bundles. When $f$ is partially hyperbolic there is a further refinement $E^{cs}=E^s\oplus E^c$ into $Df$ invariant bundles called the stable and center bundles, respectively. If furthermore $E^c$ is trivial then $f$ is said to be hyperbolic or Anosov.

From the theory of partial hyperbolicity we will only use the following well known facts (Cf. \cite{HPS}): the bundle $E^u$ integrates to an $f$ invariant foliation $W^u$, the unstable foliation. If $d_u$ denotes the intrinsic distance in a given leaf of $W^u$, it holds
\[
	\forall x,y\in M, y\in W^u(x), \forall n\geq 0, d_u(f^n(x),f^n(y))\geq \lambda^n d_u(x,y).
\]  

\section{Bounding the Lyapunov exponents from below}\label{sec.detecting}

We now establish a technical Theorem that will be used to prove the existence of large positive Lyapunov exponents for cocycles that are expanding in some large but non-invariant part of the phase space. In pursue of versatility, the result is written in some general setting. The method on the proof is based on \cite{NUHD} and its subsequent refinement \cite{RandomProdSt}. 

\smallskip

For the rest of the section we consider the following data: 
\begin{itemize}
	\item $(M,d_M)$ is a compact metric space.
	\item $\mu$ is a Borel probability measure on $M$.
	\item $f:M\to M$ is a continuous endomorphism that preserves $\mu$.
	\item $\xi$ is an increasing measurable partition for $f$.
	\item $\pi:E\to M$ is a continuous two dimensional vector sub-bundle of $\mathbb{R}^N$, equipped with a Riemannian metric whose corresponding norm is denoted $\norm{\cdot}$.
	\item $(f,A)$ is a continuous cocycle, and $(f,\PA)$ is the corresponding projective cocycle.
    \item $\mathcal{X}$ is a family of unit length vector fields over atoms of $\xi$, invariant under the cocycle. 
\end{itemize}

\begin{definition}\label{def:adapted}
We say that the invariant family of vector fields $\mathcal{X}$ over the increasing partition $\xi$ is adapted to the cocycle $(f,A)$ if there exist $\beta,\delta\in(0,1), \lambda>\| A^{-1}\|^{-1}$, a partition $\mathcal{X}=\mathcal{X}^g\cup\mathcal{X}^b$ and a subset $G\subset M$ satisfying the following conditions.

	\begin{enumerate}
        \item[\referencia{item.condicion1}{\rm H1}] For each $\p_0\in\xi$ we have
   		\[
		\mc\left(\cup\left\{\p\in\xi^1|\p_0:\ \p\subset G\right\}\right)>1- \delta.
   		\]
         
       \smallskip 

       \item[\referencia{item.condicion2}{\rm H2}]  For every $\p_0\in\xi$ there exists at least one $X:\p_0\to E\in \mathcal{X}^g$.

       \smallskip 
        
       \item[\referencia{item.condicion3}{\rm H3}] If $X:\p_0\to E\in\mathcal{X}^g$ and $m\in G$ then it holds $\norm{(A(m)\cdot X(m)}\geq \lambda$.
                
       \smallskip

    	\item[\referencia{item.condicion4}{\rm H4}] If $X:\p_0\to E\in \mathcal{X}^g$, then for every $\p\in \xi^{1}|\p_0$, $\p\subset G$ it holds $Y^1_{\p}\in \mathcal{X}^g$. 

    	\smallskip	
   		    
    	\item[\referencia{item.condicion5}{\rm H5}] If $X:\p_0\to E\in\mathcal{X}^b$, then 
		\[
			\mc\left(\cup\left\{\p\in \xi^1|\p_0: Y^1_{\p}\in \mathcal{X}^g\right\}\right)> \beta.
   		\]
    \end{enumerate}	
\end{definition}

Loosely speaking, an adapted family is a collection of (unit) vector fields with some additional property (say, regularity) that is invariant under the cocycle, and morever there is a large sub-collection $\mathcal{X}^g$ of these fields (the \emph{good} ones) that are uniformly expanded under the action $(f,A)$. This sub-family is not necessarily invariant, but for a given $X\in\mathcal{X}^g$ its image under $(f,A)$ consists mostly of vectors in $\mathcal{X}^g$. This follows from the conditions \ref{item.condicion1} and \ref{item.condicion4}. The condition \ref{item.condicion1} is in fact only on the partition, the most part of every atom is in the good region $G$, and the conditions \ref{item.condicion3} and \ref{item.condicion4} say that a good vector in a good region is expanded and goes to a good vector. Finally, there is also a restitution mechanism for vector fields outside $\mathcal{X}^g$ (condition \ref{item.condicion5}).

\smallskip

The main theorem in this part is the following.

\begin{theorem}\label{teo.tecnico}
Suppose that $\mu$ is ergodic and assume that there exists an adapted family for the cocycle $A$ with $\lambda,\delta$ as in the definition \ref{def:adapted}. Then the largest Lyapunov exponent of $(f,A,\mu)$ satisfies 
\[\chi^+>\frac {\beta}{\beta+\delta}\log\frac{\lambda^{1-\delta}}{\norm{A^{-1}}^{\delta+\frac{\delta}{\beta}}}=\log\lambda-\frac{\beta\delta+\delta}{\beta+\delta}\log(\lambda\|A^{-1}\|).
\]
\end{theorem}

\smallskip

The proof of the previous theorem will rely on the Proposition \ref{pro.liminf} below. For $X:\p\to M$ vector field and $n\in \Nat$ define
\[
I_n(X):=\int_{\p}\log\norm{A^{(n)}(m)\cdot X(m)}d\mc(m)
\]

\begin{proposition}\label{pro.liminf}
Assume that there exists $C>0$ satifying: for every atom $\p\in\xi$ there exists $X^{\p}:\p\to E$ such that
\[
	\liminf_{n\to\oo} \frac{I_n(X^{\p})}{n}\geq C.
\] 
Then $\chi^{+}\geq C$.
\end{proposition}

\begin{proof}
Denote by $X$ the (measurable) vector field obtained by gluing all the $X^{\p}$. Since $\mu$ is ergodic, we get that
for $\mu-a.e.(m)$,
\[
	\chi^+\geq \lim_{n\to\oo}\frac{\log\norm{A^{(n)}(m)\cdot X(m)}}{n}
\] 
and thus,
\begin{align*}
	\chi^+&\geq \int \lim_{n\to\oo}\frac{\log\norm{A^{(n)}(m)\cdot X_m}}{n}d\mu(m)=\int_{M/\xi} \Big(\int_{\xi} \lim_{n\to\oo}\frac{\log\norm{A^{(n)}(m)\cdot X_m}}{n}  d\mc(m)\Big)d\mu_{M/\xi}(\xi)\\
	& =\int_{M/\xi}\Big(\lim_{n\to\oo}\frac{1}{n}\int_{\xi}\log\norm{A^{(n)}(m)\cdot X_m} d\mc(m)\Big)d\mu_{M/\xi}(\xi)\geq \int_{M/\xi} Cd\mu_{M/\xi}(\xi)=C.
\end{align*}
\end{proof}

We use this Proposition in conjunction with the following Lemma.

\begin{lemma}\label{lem.Insuma}
Let $X:\p_0\to E$ be a unit vector field. Then for every $n\in\mathbb{N}$ we can write
	\[
	I_n(X)=\sum_{i=0}^{n-1}\sum_{\p\in \xi^i|\p_0}\mc(\p)\int_{f^i(\p)} \log\norm{A(m)\cdot Y^i_{\p}(m)}\mc(m)
	\]
	where $Y^i_{\p}$ are the unit vector field obtained by pushing forward $X$.
\end{lemma}

\begin{proof}
	We proceed by induction in $n$. For $n=1$ is just the definition, so assuming the claim for $n$, we compute for $n+1$
	\begin{align*}
	&I_{n+1}(X)=\int_{\p_0} \log\norm{A^{(n+1)}(m)\cdot X(m)}d\mc(m)=\int_{\p_0} \log\norm{A(f^nm)\cdot A^{(n)}(m)\cdot X(m)}d\mc(m)\\
	&=\int_{\p_0} \log\norm{A(f^nm)\cdot \frac{A^{(n)}(m)\cdot X(m)}{\norm{A^{(n)}(m)\cdot X(m)}}}d\mc(m)+I_n(X)=\\
	&\int_{\p_0} \log\norm{A(f^nm)\cdot \frac{A^{(n)}(m)\cdot X(m)}{\norm{A^{(n)}(m)\cdot X(m)}}}d\mc(m)+\cramped{\sum_{i=0}^{n-1}\sum_{\scriptscriptstyle \p\in \xi^i|\xi(p_0)}}\mc(\p)\int_{f^i(\p)} \log\norm{A(m)\cdot Y^i_{\p}(m)}\mc(m)
	\end{align*}
	so it suffices to study the integral above. Since $\mu^{\xi^{n+1}}|\p$ is just the conditional measure of $\mc$ in this atom, we can write
	\begin{align*}
	&\int_{\p_0} \log\norm{A(f^nm)\cdot \frac{A^{(n)}(m)\cdot X(m)}{\norm{A^{(n)}(m)\cdot X(m)}}}d\mc(m)=\sum_{\p\in \xi^{n+1}|\p_0}
	\int_{\p} \log\norm{A(f^nm)\cdot \frac{A^{(n)}(m)\cdot X(m)}{\norm{A^{(n)}(m)\cdot X(m)}}}d\mc(m)\\
	&=\sum_{\mathclap{\p\in \xi^{n+1}|\p_0}} \mc(\p)\int_{\p} \log\norm{A(f^nm)\cdot Y^{n+1}_{\p}(f^{n}m)}d\mu^{\xi^{n+1}}
	(m)
	\end{align*}
	where 
	\begin{align*}
	Y^{n+1}_{\p}(m')=\frac{A^{(n)}(g^{n+1}_{\p}m')\cdot X(g^{n}_{\p}m')}{\norm{A^{(n)}(g^{n}_{\p}m')\cdot X(g^{n}_{\p}m')}}\quad m'\in f^n(\p).
	\end{align*} 
	Since $f_{\ast}^n\mu^{\xi^{n+1}}|\p=\mu^{\xi}|f^{n+1}(\p)$ we end up getting
	\begin{multline*}
	\int_{\p_0} \log\norm{A(f^nm)\cdot \frac{A^{(n)}(m)\cdot X(m)}{\norm{A^{(n)}(m)\cdot X(m)}}}d\mc(m)=
	\sum_{\mathclap{\p\in \xi^{n+1}|\p_0}} \mc(\p)\int_{f^{n}(\p)} \log\norm{A(fm)\cdot Y^{n+1}_{\p}(m)}d\mc(m)
	\end{multline*}
	thus concluding the induction step.
\end{proof}

\begin{definition}
	For a unit vector field $X:\p\to E$  we write
	\[E(X):=\int_{\p} \log\norm{A(m)\cdot X(m)}d\mc(m)
	\]
\end{definition}
Note that the previous lemma shows that if $X\in\mathcal{X}$ and $\p\in\xi$ then for every $n\geq 1$ the quantity $\frac{I_n(X)}{n}$ can be written
as a convex combination of $\{E(Y)\}_{Y\in \mathcal{F}_n(X)}$ where $\mathcal{F}_n(X)\subset \mathcal{X}$ is finite. 

\smallskip

The following is immediate from hypotheses \ref{item.condicion1}, \ref{item.condicion3}.

\begin{lemma}\label{lem.estimativasE}
For $X:\p\to E\in\mathcal{X}$ it holds
\begin{itemize}
\item $E(X)\geq -\log\norm{A^{-1}}$.
\item If furthermore $X\in \mathcal{X}^g$ then $E(X)\geq (1-\delta)\log\lambda-\delta\log \norm{A^{-1}}$.
\end{itemize}
\end{lemma}

\begin{definition}
	For $X:\p_0\to E\in\mathcal{X}$ we denote
	\begin{align}
	&g_n=g_n(X)=\mc\left(\bigcup\{\p\in \xi^{n}|\p_0: Y^n_{\p}:f^n(\p)\to E\in\mathcal{X}^g\}\right)\\
	&b_n=b_n(X)=\mc\left(\bigcup\{\p\in \xi^{n}|\p_0: Y^n_{\p}:f^n(\p)\to E\in\mathcal{X}^b\right)	
	\end{align}	
\end{definition}

By hypotheses \ref{item.condicion4},\ref{item.condicion5} we deduce directly.

\begin{lemma}
Assume that $X:\p_0\to E\in \mathcal{X}^g$. Then or every $n\geq 0$,
\begin{align*}
&g_{n+1}> (1-\delta)\cdot g_n+\beta\cdot b_n	\\
&b_{n+1}< \delta\cdot g_n+(1-\beta)\cdot b_n	
\end{align*}	
\end{lemma}

\begin{proof}	
This is established by induction. The base case will follow from the inductive argument together with the hypothesis that $X\in\mathcal{X}^g$. Fix then $n$ and note that as a consequence of \ref{item.condicion4},\ref{item.condicion5}, each atom $\p\in\xi^n|$ for which $Y^n_{\p}:f^n(\p)\to E$ is in $\mathcal{X}^g$ can be partitioned into atoms $\p'$ of $\xi^{n+1}$ such that
\[
	\mc\left(\p'\in\xi^{n+1}|\p:Y^n_{\p}:f^{n+1}(\p')\to E\in\mathcal{X}^b\right)< \delta.
\]
This implies that the contribution of $g_n$ to $g_{n+1}$ is $\geq 1-\delta$. On the other hand, considering a worse case scenario, each atom $\p\in \xi^n$ for which $Y^n_{\p}:f^n(\p)\to E$ is in $\mathcal{X}^b$ contributes to $g_{n+1}$ by an porcentage of at least $\beta$ of its $\mc$ measure. Altogether,
\[
g_{n+1}> (1-\delta)\cdot g_n+\beta\cdot b_n.
\]
The inequality for $b_{n+1}$ follows from this one.
\end{proof}

\begin{corollary}\label{cor.proporcionbuenosvsmalos}		
Assume that $X:\p_0\to E\in\mathcal{X}^g$. Then for every $n$ it holds
$$
g_n>\frac {\beta}{\beta+\delta},\ \ \ b_n<\frac{\delta}{\beta+\delta}.
$$
\end{corollary}

\begin{proof}
Since $X\in \mathcal{X}^g$, $g_0=1$. The previous Lemma shows that
$$g_{n+1}> (1-\delta)\cdot g_n,+\beta\cdot b_n=\left(1-\beta-\delta\right)g_n+\beta.$$
The first inequality follows by induction, and the second one is just because $b_n=1-g_n$.
\end{proof}	

\smallskip

We are now ready to prove Theorem \ref{teo.tecnico}. Start with a vector field $X:\p_0\to E\in \mathcal{X}^g$ and use Lemma \ref{lem.Insuma} to write

\begin{align*}
	\frac{I_n(X)}{n}&=\frac{1}{n}\sum_{i=0}^{n-1}\sum_{\p\in\xi^i|\p_0} \mc(\p)\cdot E(Y^i_{\p})
	=\frac{1}{n}\sum_{i=0}^{n-1}\Big(\sum_{\mathclap{\substack{\p\in\xi^i|\p_0\\ Y^i_{\p}\in\mathcal X^g}}} \mc(\p)\cdot E(Y^i_{\p})+\sum_{\mathclap{\substack{\p\in\xi^i|\p_0\\ Y^i_{\p}\in\mathcal X^b}}} \mc(\p)\cdot E(Y^i_{\p})\Big)\\
	&\geq \frac{1}{n}\sum_{i=0}^{n-1} g_i((1-\delta)\log\lambda-\delta\log \norm{A^{-1}})- b_i\log \norm{A^{-1}}\quad \text{ by Lemma }\ref{lem.estimativasE}\\
	&=\frac{1}{n}\sum_{i=0}^{n-1 }g_i((1-\delta)\log\lambda-\delta\log \norm{A^{-1}})- (1-g_i)\log \norm{A^{-1}}\\
	&=\frac{1}{n}\sum_{i=0}^{n-1 }g_i(1-\delta)(\log\lambda+\log\norm{A^{-1}})-\log\norm{A^{-1}}\\
	&\geq \frac{1}{n}\sum_{i=0}^{n-1} \frac {\beta}{\beta+\delta} (1-\delta)(\log\lambda+\log\norm{A^{-1}})-\log\norm{A^{-1}}\quad \text{ by Corollary }\ref{cor.proporcionbuenosvsmalos}\\
	&\geq \frac {\beta}{\beta+\delta}\log\frac{\lambda^{1-\delta}}{\norm{A^{-1}}^{\delta+\frac{\delta}{\beta}}}.
\end{align*}	
Theorem \ref{teo.tecnico} is then consequence of Proposition \ref{pro.liminf}.

\begin{remark}\label{rem.nonergodic}
The ergodicity of $\mu$ is used only in Proposition \ref{pro.liminf}. Without the ergodicity assumption on $\mu$ one only gets the same lower bound for $\int_P\chi^+(m)d\mu^{\xi}(m)$, for every atom $P$ of $\xi$. If $\xi$ is generating it is easy to see that one gets the same lower bound for the Lyapunov exponent $\chi^{+}(m)$ for $\mu$-a.e.$(m)$.
\end{remark}

\begin{remark}
Theorem \ref{teo.tecnico} is useful in the situation where there is an expanding cone for the cocycle on a large portion of the space, but that is not necessarily invariant; this is exploited in \cite{NUHD} and \cite{RandomProdSt}. In the Appendix we show how to deduce the existence of an adapted family in such case.  
\end{remark}

\section{Perturbing the Anosov maps: preparations for the proofs of Theorem A and Theorem B}
\label{sec:prelim}

We start the section with several considerations on the Anosov automorphism $L$ of $\mathbb T^3$. It is more convenient to work with an automorphism $L$ which has one unstable direction $E^u$ and two stable directions, a weak unstable direction $E^{ws}$ and a strong unstable direction $E^{ss}$. The corresponding eigenvalues are denoted $\lambda_u>1>\lambda_{ws}>\lambda_{ss}$. By eventually making a $\mathbb T^3$-preserving change of basis in $\mathbb R^3$, we can assume the following:
\begin{enumerate}
\item
The 2-dimensional stable subspace $E^s$ is close to being horizontal: it is generated by the vectors $\overline a=(1,0,a)$ and $\overline b=(0,1,b)$, with $a,b\in (0,1)$; this implies that the angle between $\overline a$ and $\overline b$ is inside the interval $(\frac {\pi}3, \frac {\pi}2)$;
\item
The angle between $\overline a$ and the weak stable unit eigenvector $v_{ws}$ is smaller than the half of the angle between $\overline a$ and $\overline b$, so it is smaller than $\frac {\pi}4$ (this can be done by switching $\overline a$ and $\overline b$, if needed);
\item
The angle between $\overline b$ and the strong stable unit eigenvector $v_{ss}$ is $\theta_0>0$;
\end{enumerate}

We now construct the map $f$ inside the $C^{\oo}$ family of diffeomorphisms $f_t:\mathbb T^3\rightarrow\mathbb T^3$ given by the formula
\begin{equation}\label{eq:ft}
f_t(x,y,z)=(x, y+t\sin 2\pi x, z+bt\sin 2\pi x)\mod\mathbb Z^3.
\end{equation}
The maps $f_t$ are clearly isotopic to the identity, preserve the volume, and also preserve the planes of the stable foliation of $L$. In fact $f_t$ is just a translation on the lines parallel to $(0,1,b)$, and restricted to the stable leaves preserves the area. Observe that $f_t^{-1}=f_{-t}$. See figure \ref{fig.ft}.

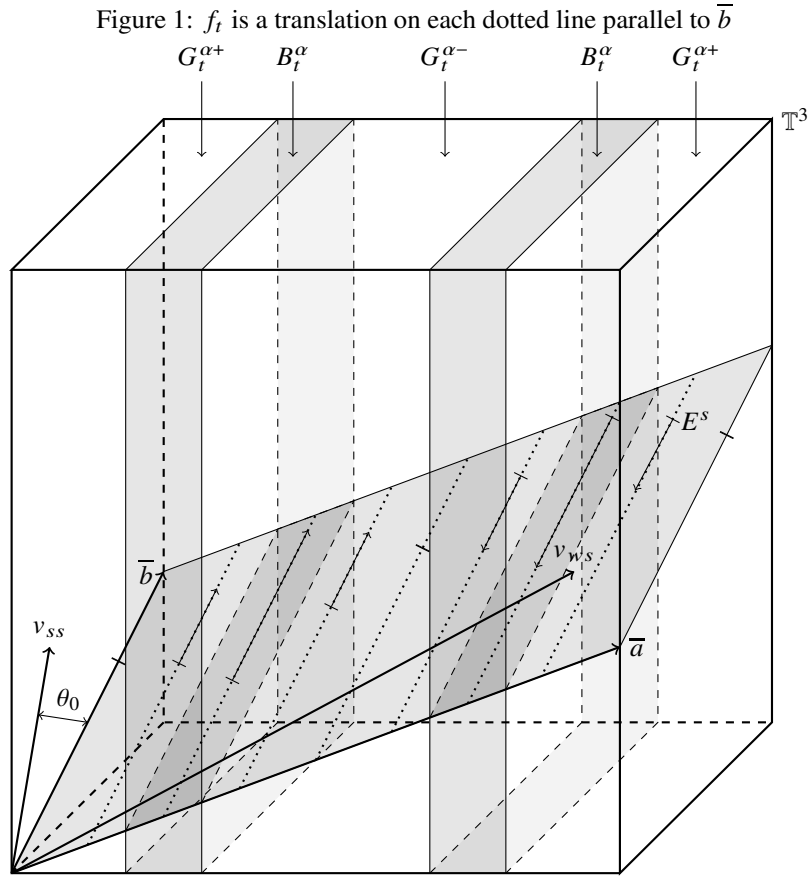
\begin{figure}[h]
\begin{center}
\caption{$f_t$ is a translation on each dotted line parallel to $\overline b$}\label{fig.ft}
\begin{tikzpicture}

\draw[black, thick] (4,4) -- (4,-4) -- (6,-2) -- (6,6) -- (4,4) -- (-4,4) -- (-2,6) -- (6,6) node[anchor=west] {$\mathbb T^3$};
\draw[black, thick] (-4,4) -- (-4,-4) -- (4,-4);
\draw[black, thick, dashed] (-2,6) --(-2,-2) -- (6,-2);
\draw[black, thick, dashed] (-2,-2) --(-4,-4);
\draw[black, thick, dotted] (-3,-3.625) --(-1,0.375);
\draw[black, thin, |->] (-1.8,-1.225) -- (-1.3,-0.225);
\draw[black, thick, dotted] (-2,-3.25) --(0,0.75);
\draw[black, thin, |->] (-1.1,-1.45) -- (-0.1,0.55);
\draw[black, thick, dotted] (-1,-2.875) --(1,1.125);
\draw[black, thin, |->] (0.2,-0.475) -- (0.7,0.525);
\draw[black, thick, dotted] (0,-2.5) --(2,1.5);
\draw[black, thin, |-|] (1.4,0.3) -- (1.41,0.32);
\draw[black, thick, dotted] (1,-2.125) --(3,1.875);
\draw[black, thin, |->] (2.7,1.275) -- (2.2,0.275);
\draw[black, thick, dotted] (2,-1.75) --(4,2.25);
\draw[black, thin, |->] (3.9,2.05) -- (2.9,0.05);
\draw[black, thick, dotted] (3,-1.325) --(5,2.625);
\draw[black, thin, |->] (4.7,2.025) -- (4.2,1.075);
\draw[black, thin, |-|] (-2.6,-1.2) -- (-2.59,-1.18);
\draw[black, thin, |-|] (5.4,1.8) -- (5.41,1.82);
\filldraw[black, thin, fill opacity=0.1]  (-2.5,-4) -- (-2.5,4) -- (-0.5,6) -- (0.5,6) -- (-1.5,4) -- (-1.5,-4);
\filldraw[black, thin, dashed, fill opacity=0.05] (-2.5,-4) -- (-0.5,-2) -- (-0.5,6) -- (0.5,6) -- (0.5,-2) -- (-1.5,-4);
\draw[black, thin, <-] (-0.3,5.5) -- (-0.3,6.5) node[anchor=south] {$B_t^{\alpha}$};
\filldraw[black, thin, fill opacity=0.1] (1.5,-4) -- (1.5,4) -- (3.5,6) -- (4.5,6) -- (2.5,4) -- (2.5,-4);
\filldraw[black, thin, dashed, fill opacity=0.05] (1.5,-4) -- (3.5,-2) -- (3.5,6) -- (4.5,6) -- (4.5,-2) -- (2.5,-4);
\draw[black, thin, <-] (3.7,5.5) -- (3.7,6.5) node[anchor=south] {$B_t^{\alpha}$};
\draw[black, thin, <-] (1.7,5.5) -- (1.7,6.5) node[anchor=south] {$G_t^{\alpha-}$};
\draw[black, thin, <-] (-1.5,5.5) -- (-1.5,6.5) node[anchor=south] {$G_t^{\alpha+}$};
\draw[black, thin, <-] (5,5.5) -- (5,6.5) node[anchor=south] {$G_t^{\alpha+}$};
\filldraw[black, thin, dashed, fill opacity=0.1] (-2.5,-3.4375) -- (-0.5,0.5625) -- (0.5,0.9375) -- (-1.5,-3.0625);
\filldraw[black, thin, dashed, fill opacity=0.1] (1.5,-1.9375) -- (3.5,2.0625) -- (4.5,2.4375) -- (2.5,-1.5625);
\draw[black, thin, <->] (-3,-2) -- (-3.65,-1.9);
\node at (-3.25,-1.7) {$\theta_0$};
\draw[black, thick,->] (-4,-4) -- (4,-1) node[anchor=west] {$\overline a$};
\draw[black, thick,->] (-4,-4) -- (-2,0) node[anchor=east] {$\overline b$};
\filldraw[black, thin, fill opacity=0.1] (-4,-4) -- (4,-1) -- (6,3) -- (-2,0);
\draw[black, thick,->] (-4,-4) -- (3.4,0) node[anchor=south] {$v_{ws}$};
\draw[black, thick,->] (-4,-4) -- (-3.5,-1) node[anchor=south] {$v_{ss}$};
\node at (5,2) {$E^s$};

\end{tikzpicture}
\end{center}
\end{figure}

From now on we assume that $t>1$ and we denote by $C_L$ some positive constant which depends only on $L$ (and is independent of $t$ and $n$); this constant will be updated as needed, and may depend on the considered equality/inequality. If some value of it needs to be specified, it will be denoted by another letter. The derivative of $f_t$ is
\begin{equation}
Df_t(x,y,z)=\left[ \begin{array}{ccc}
1 & 0 & 0 \\
2\pi t\cos 2\pi x & 1 & 0 \\
2\pi tb\cos 2\pi x & 0 & 1 \end{array} \right]
\end{equation}

We have that $\| Df_t\|,\ \| Df_t^{-1}\|> C_Lt$ and $b(Df_t)< C_Lt^2$. We denote by $\mathbb PDf_t$ the projectivization of $Df_t|_{E^s}$ (it acts on the circle of unit stable vectors). A simple calculation shows that $\mathbb PDf_t(p)(v)$ is Lipschitz with Lipschitz constant $C_Lt^2$ in both variables (the base point $p$ and the unit vector $v$).

Observe that $Df_t$ preserves $E^s$ and
\begin{eqnarray*}
Df_t(x,y,z)\overline b&=&\overline b\\
Df_t(x,y,z)\overline a&=&\overline a+2\pi t\cos 2\pi x\overline b.
\end{eqnarray*}
If $v_u$ is the unstable unit eigenvector, then we have
\[
Df_tv_u=v_u+\tilde v_s,\ \ \ \tilde v_s\in E^s,\ \ \ \|\tilde v_s\|< C_Lt.
\]

We now investigate the Lyapunov exponents of $L^n\circ f_t$.

\subsection{The new unstable foliation}

We consider the unstable cone of size $\gamma$:
\[
C^u_{\gamma}=\mathbb R\cdot\{ v_u+w_s:\ \|w_s\|<\gamma\}.
\]
The following is an estimate of the size of an invariant unstable cone for $L^n\circ f_t$.

\begin{lemma}
There exists $\gamma_L>0$ such that if
\begin{equation}\label{eq:gamma}
\gamma=\gamma_Lt\cdot\frac{\lambda_{ws}^n}{\lambda_u^n}<1
\end{equation}
then the unstable cone $C^u_{\gamma}$ is invariant under $L^n\circ f_t$.
\end{lemma}

\begin{proof}

Let $v\in C_{\gamma}^u$, $v=v_u+w_s$, $\|w_s\|<\gamma$. Then
\[
Df_tv=Df_tv_u+Df_tw_s=v_u+\tilde v_s+Df_tw_s:=v_u+v_s',
\]
where
\[
\|v_s'\|=\|\tilde v_s+Df_tw_s\|\leq \|\tilde v_s\|+\|Df_t\|\cdot\|w_s\|<C_Lt+C_Lt\gamma=C_Lt(1+\gamma).
\]
Then
$$
L^nDf_tv=L^n(v_u+v_s')=\lambda_u^nv_u+L^nv_s':=\lambda_u^n(v_u+\tilde w_s),
$$
where
\[
\|\tilde w_s\|=\frac 1{\lambda_u^n}\|L^nv_s'\|<C_Lt\frac{\lambda_{ws}^n}{\lambda_u^n}(1+\gamma).
\]

In consequence, if $C_Lt\cdot\frac{\lambda_{ws}^n}{\lambda_u^n}(1+\gamma)<\gamma$ then $C_{\gamma}^u$ is invariant under $L^nDf_t$. Take $\gamma_L=2C_L$ and the conclusion follows.

\end{proof}

\begin{corollary}
If $\gamma_Lt\cdot\frac{\lambda_{ws}^n}{\lambda_u^n}<1$ then $L^n\circ f_t$ is partially hyperbolic. Its center stable bundle $E^{cs}_{t,n}$ coincides with the stable bundle $E^s$ of $L$, and therefore $E^{cs}_{t,n}$ integrates to a foliation (center-stable foliation) $W^{cs}_{t,n}$ that coincides with the stable foliation $W^s$ of $L$. The angle between the unstable bundle $E^u_{t,n}$ of $L^n\circ f_t$ and $E^u$, the unstable bundle of $L$, is bounded from above by the number $\gamma$ given by the formula \ref{eq:gamma}. In particular, the expansion rate along the unstable foliation $W^u_{t,n}$ of $L^n\circ f_t$ lies inside the interval \(\left(\lambda_u^n\cdot\frac{1-\gamma}{1+\gamma},\lambda_u^n\cdot\frac{1+\gamma}{1-\gamma}\right)\).
\end{corollary}

\begin{definition}\hypertarget{condPH}{}
If the relation $\gamma_Lt\cdot\frac{\lambda_{ws}^n}{\lambda_u^n}<1$ is satisfied we say that we have the {\bf condition (PH)}, and $L^n\circ f_t$ is partially hyperbolic.
\end{definition}

Let us remark that even though the expansion rate along the unstable foliation $W^u_{t,n}$ has the lower bound $\lambda_u^n\cdot\frac{1-\gamma}{1+\gamma}$, in fact the exponential rate of expansion is $\lambda_u^n$. This is because we have the relation
\[
\lambda_u^{kn}\frac{1-\gamma}{1+\gamma}< \|D(L^n\circ f_t)^k|_{E^u_{t,n}}\|<\lambda_u^{kn}\frac{1+\gamma}{1-\gamma}.
\]
On the other hand we have the bound for the expansion on the center-stable bundle
\[
\|D(L^n\circ f_t)|_{E^{s}}\|\leq\|L^n|_{E^{s}}\|\cdot\|Df_t|_{E^{s}}\|< C_Lt\lambda_{ws}^n.
\]
In other words, there exists some fixed value of $C_L$, which from now will be denoted by $a_L$, such that $\|D(L^n\circ f_t)|_{E^{s}}\|< a_Lt\lambda_{ws}^n$.

\begin{definition}\hypertarget{condA}{}
If the relation $a_Lt\lambda_{ws}^n<1$ is satisfied we say that we have the {\bf condition (A)}, and $L^n\circ f_t$ is Anosov.
\end{definition}

Now we turn our attention to the disintegration of the volume along $W^u_{t,n}$, the unstable foliation of $L^n\circ f_t$. It is well known that this foliation is absolutely continuous (see for example \cite{ErgAnAc}), and the disintegration of the volume along the unstable leaves have densities $\rho$ which satisfy the following formula:
\[
\frac{\rho(x)}{\rho(y)}=\lim_{k\rightarrow\infty}\frac{\left|D(L^n\circ f_t)^{-k}|_{E^u_{t,n}(x)}\right|}{\left|D(L^n\circ f_t)^{-k}|_{E^u_{t,n}(y)}\right|},
\]
for $x$ and $y$ in the same unstable leaf. 

Let $v_u+v_u'(x)$ be an unstable vector in $E^u_{t,n}(x)$, with $v_u'(x)\in E^s$. Then 
\[
D(L^n\circ f_t)^{-k}(v_u+v_u'(x))=C_x\left(v_u+v_u'((L^n\circ f_t)^{-k}(x))\right).
\]
Since $L$ and $Df_t$ preserve the $u$-component of vectors, we see that $C_x=\lambda_u^{-kn}$, so
\[
\left|D(L^n\circ f_t)^{-k}|_{E^u_{t,n}(x)}\right|=\lambda_u^{-kn}\frac{\|v_u+v_u'((L^n\circ f_t)^{-k}(x))\|}{\|v_u+v_u'(x)\|}.
\]
A similar formula holds for $y$. Since $d\left((L^n\circ f_t)^{-k}(x),(L^n\circ f_t)^{-k}(y)\right)\xrightarrow[k\to\oo]{} 0$, we obtain that
\[
\frac{\rho(x)}{\rho(y)}=\frac{\|v_u+v_u'(x)\|}{\|v_u+v_u'(y)\|}.
\]
This implies the following results.

\begin{lemma}
If the condition \hyperlink{condPH}{(PH)} is satisfied, then the disintegrations of the volume along the unstable foliation is exactly (modulo re-scaling) the pull-back of the Lebesgue measure for the projection of $W^u_{t,n}$ into the original $W^u$ parallel to $E^u$. The densities of disintegrations satisfy the bounds
\begin{equation}\label{eq:density}
\frac{1-\gamma}{1+\gamma}<\frac{\rho(x)}{\rho(y)}<\frac{1+\gamma}{1-\gamma},\ \ \forall x,y,\ y\in W^u_{t,n}(x).
\end{equation}
\end{lemma}

\subsection{Estimates on the derivative}

Inside $E^s$ we define {\it the good cone $C_g$}:
\begin{equation}
C_g=\{ r\overline a+s\overline b:\ |s|<3|r|\}.
\end{equation}
Let us remark that the good cone is quite big, and it contains $\overline a$ and $E^{ws}$, while its complement is smaller, centered around $\overline b$.

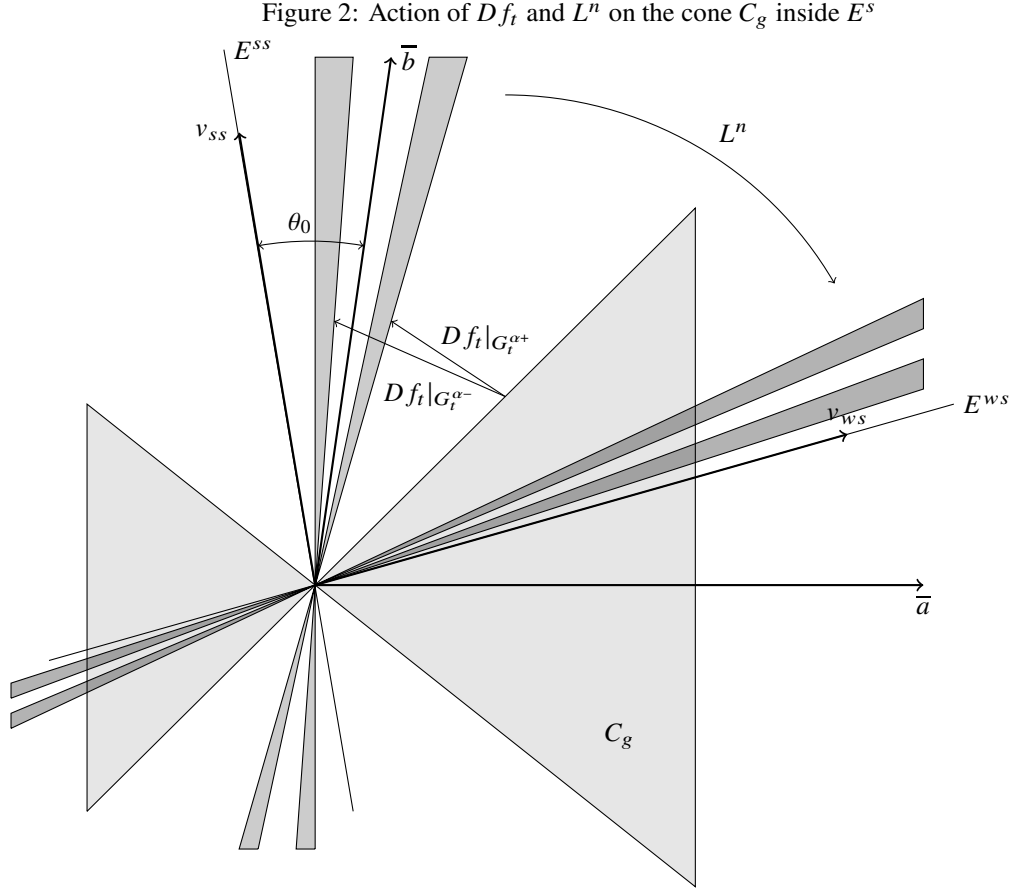
\begin{figure}[H]\caption{Action of $Df_t$ and $L^n$ on the cone $C_g$ inside $E^s$}
\begin{tikzpicture}

\draw[black, thin, <->] (-2.75,2.5) arc (100:80:4);
\node at(-2.2,2.8) {$\theta_0$};
\draw[black, thick,->] (-2,-2) -- (-1,5) node[anchor=west] {$\overline b$};
\draw[black, thick,->] (-2,-2) -- (6,-2) node[anchor=north] {$\overline a$};
\filldraw[black, thin, fill opacity=0.1] (-5,-5) -- (3,3) -- (3,-6) -- (-5,0.4) -- (-5,-5);
\draw[black, thick,->] (-2,-2) -- (5,0) node[anchor=south] {$v_{ws}$};
\draw[black, thick,->] (-2,-2) -- (-3,4) node[anchor=east] {$v_{ss}$};
\draw[black, thin] (-5.5,-3) -- (6.4,0.4) node[anchor=west] {$E^{ws}$};
\draw[black, thin] (-1.5,-5) -- (-3.2,5.1) node[anchor=west] {$E^{ss}$};
\node at (2,-4) {$C_g$};
\filldraw[black, thin, fill opacity=0.2] (-2,-5.5) -- (-2,5) -- (-1.5,5) -- (-2.25,-5.5) -- (-2,-5.5);
\filldraw[black, thin, fill opacity=0.2] (-2.75,-5.5) -- (-0.5,5) -- (0,5) -- (-3,-5.5) -- (-2.75,-5.5);
\draw[black, thin, ->] (0.5,0.5) -- (-1,1.5);
\node at (0.25,1.25) {$Df_t|_{G_t^{\alpha+}}$};
\draw[black, thin, ->] (0.5,0.5) -- (-1.75,1.5);
\node at (-0.5,0.5) {$Df_t|_{G_t^{\alpha-}}$};
\filldraw[black, thin, fill opacity=0.3] (-6,-3.5) -- (6,1) -- (6,0.6) -- (-6,-3.3) -- (-6,-3.5);
\filldraw[black, thin, fill opacity=0.3] (-6,-3.9) -- (6,1.8) -- (6,1.4) -- (-6,-3.7) -- (-6,-3.9);
\draw[black, thin, ->] (0.5,4.5) arc (90:30:5);
\node at(3.5,4) {$L^n$};

\end{tikzpicture}
\end{figure}

From now on we assume that $t>1$. Let $\alpha\in(0,\frac 12)$. {\it The bad region $B_t^{\alpha}$} on the torus is
\begin{equation}
B_t^{\alpha}=\{ (x,y,z)\in\mathbb T^3:\ |\cos 2\pi x|<t^{-\alpha}\}.
\end{equation}
The complement of the bad region is {\it the good region $G_t$}. We divide it into two parts:
\[
G_t^{\alpha+}=\{ (x,y,z)\in\mathbb T^3:\ \cos 2\pi x>t^{-\alpha}\}\ \ G_t^{\alpha-}=\{ (x,y,z)\in\mathbb T^3:\ \cos 2\pi x<-t^{-\alpha}\}.
\]
The three regions $B_t^{\alpha}, G_t^{\alpha+}$ and $G_t^{\alpha-}$ are invariant by $f_t$ because $f_t$ does not change the $x$-coordinate. We remark that for large $t$ the bad region is small, it has size of the order of $t^{-\alpha}$, and the two good regions have equal sizes, close to $\frac 12$.

Let us see the action of $Df_t$ on vectors in $E^s$.

\begin{lemma}
Let $v\in E^s$. 
\begin{enumerate}
\item For $t>1$, $v\in C_g$ and $p\in G_t^{\alpha}$ implies
\begin{equation}
\frac{\|Df_t(p)v\|}{\|v\|}> C_Lt^{1-\alpha}.
\end{equation}
Furthermore there exists $t_0>0$ such that if $t>t_0$ then the angle between $Df_t(p)v$ and $\overline b$ is smaller than $\frac{\theta_0}2$.
\item
If $v\in C_g$ and $p\in B_t^{\alpha}$, or if $v\notin C_g$, then
\begin{equation}
\frac{\|Df_t(p)v\|}{\|v\|}> \frac {C_L}t.
\end{equation}
\end{enumerate}
\end{lemma}

\begin{proof}
Take first $v=\overline a+r\overline b=(1,r,a+rb)$, $r<3$, so $v\in C_g$ and $\|v\|$ is bounded by some constant $R_L$ only depending on $L$. Applying $Df_t(p)$ for $p=(x,y,z)\in G_t^{\alpha}$ and using the fact that $|\cos 2\pi x|>t^{-\alpha}$ on $G_t^{\alpha}$ we get
\[
\|Df_t(p)v\|=\|(1,2\pi t\cos 2\pi x+r,2\pi tb\cos 2\pi x+a+rb)\|>2\pi(t^{1-\alpha}-1)\geq \frac{2\pi}{R_L}(t^{1-\alpha}-1)\|v\|
\]
which for $t>1$ can be written as
\[
	\|Df_t(p)v\|\geq C_Lt^{1-\alpha}.
\]
Remember that $Df_t(p)\overline a=\overline a +2\pi t\cos 2\pi x\overline b$ and $Df_t(p)\overline b=\overline b$ so
\[
Df_t(p)v=Df_t(p)(\overline a+r\overline b)=\overline a+(r+2\pi t\cos 2\pi x)\overline b.
\]
Observe that the $\overline b$ component of $Df_t(p)v$ satisfies $|r+2\pi t\cos 2\pi x|\geq 2\pi t^{1-\alpha}-1> \pi t^{\frac 12}$ if $t>1$. Then the second part of (1) follows. 

For the second part observe that 
\begin{equation}
Df_t^{-1}(x,y,z)=\left[ \begin{array}{ccc}
1 & 0 & 0 \\
-2\pi t\cos 2\pi x & 1 & 0 \\
-2\pi tb\cos 2\pi x & 0 & 1 \end{array} \right]=Df_{-t}(x,y,z)
\end{equation}
and therefore, for $p=(x,y,z)$
\[
	\inf_{v}\{\norm{D_pf_tv}:\norm{v}=1\}=\frac{1}{\norm{D_pf_t^{-1}}}\geq \frac{1}{1+4\pi t|\cos 2\pi x|}\geq \frac{C_L}{t} 
\]
if $t>1$. This implies the second part.
\end{proof}

Now let us study the action of $L^n$ on $E^s$.

\begin{lemma}
Let $v\in E^s$. Then
\begin{enumerate}
\item
There exists $n_0>0$ such that if $n>n_0$ and the angle between $v$ and $v_{ss}$ is greater than $\frac{\theta_0}2$ then we have $L^nv\in C_g$ and
\begin{equation}\label{eq:expAnC}
\frac{\|L^nv\|}{\|v\|}> C_L\lambda_{ws}^n.
\end{equation}
\item
If the angle between $v$ and $v_{ss}$ is smaller or equal to $\frac{\theta_0}2$ then
\begin{equation}
\frac{\|L^nv\|}{\|v\|}> C_L\lambda_{ss}^n.
\end{equation}
\end{enumerate}
\end{lemma}

\begin{proof}
First take $v\in E^s$, $v=v_{ws}+rv_{ss}$ with $|r|<\frac{C_L}{\theta_0}$ such that the angle between $v$ and $v_{ss}$ is greater that $\frac{\theta_0}2$. Then $\|v\|< \frac{C_L}{\theta_0}+1$. Applying $L^n$ we obtain 
\[
\|L^nv\|=\|\lambda_{ws}^nv_{ws}+\lambda_{ss}^nrv_{ss}\|> \lambda_{su}^n\left(1-\frac{C_l}{\theta_0}\frac{\lambda_{ss}^n}{\lambda_{ws}^n}\right).
\]
This will give us \ref{eq:expAnC} provided that $\frac{\lambda_{ss}^{n_0}}{\lambda_{ws}^{n_0}}\leq \frac{\theta_0}{2C_L}$.

On the other hand observe that
\[
\frac{L^nv}{\lambda_{ws}^n}=v_{ws}+r\frac{\lambda_{ss}^n}{\lambda_{ws}^n}v_{ss},
\]
so the angle between $L^nv$ and $v_{ws}$ is small whenever $\left|r\frac{\lambda_{ss}^n}{\lambda_{ws}^n}\right|$ is small. In particular if we choose $n_0$ such that
\[
|r|\frac{\lambda_{ss}^n}{\lambda_{ws}^n}<\frac{C_L}{\theta_0}\frac{\lambda_{ss}^{n_0}}{\lambda_{ws}^{n_0}}<C_Ld_{\mathbb P}(v_{ws},C_g)
\]
then $L^nv\in C_g$.

The second part is obvious.
\end{proof}

Putting together the two lemmas above we get the following proposition which describes the action of $f_t\circ L^n$ on $E^s$.

\begin{proposition}\label{prop:expansion}
Assume that the conditions $n>n_0$ and $t>t_0$ are satisfied.  
\begin{enumerate}
\item If $v\in C_g$ and $p\in G_t^{\alpha}$, then
\begin{equation}\label{eq:strong}
\frac{\|D(L^n\circ f_t)(p)v\|}{\|v\|}> C_L\lambda_{ws}^nt^{1-\alpha}
\end{equation}
and $D(L^n\circ f_t)(p)v\in C_g$.
\item For every $p\in\mathbb T^3$ and $v\neq 0$,
\begin{equation}\label{eq:weak}
\frac{\|D(L^n\circ f_t)(p)v\|}{\|v\|}> C_L\frac {\lambda_{ss}^n}t.
\end{equation}
\end{enumerate}
\end{proposition}

\subsection{Separation of preimages of the bad cone}

Now we investigate the preimages of the bad cone $C_g^c:=C_b$. We will consider the preimages on the two good regions $G_t^{\alpha+}$ and $G_t^{\alpha-}$:
\begin{align*}
C^+&=\cup_{(L^n\circ f_t)^{-1}(p)\in G_t^{\alpha+}}D(L^n\circ f_t)^{-1}(p)C_b\\
C^-&=\cup_{(L^n\circ f_t)^{-1}(p)\in G_t^{\alpha-}}D(L^n\circ f_t)^{-1}(p)C_b.
\end{align*}
$L^{-n}$ pushes the bad cone close to $E^{ss}$. Then the preimage under $Df_t$ pushes the bad cone in one direction on $G_t^{\alpha+}$, and in the other direction on $G_t^{\alpha-}$, thus creating a separation between the two preimages. We have the following result.
\begin{lemma}\label{le:separation}
There exist $t_1\geq t_0$, $n_1\geq n_0$ and $s_L>0$ such that, if $t>t_1$ and $n>n_1$ then
\begin{equation}
d_{\mathbb P}(C^+,C^-)>s_Lt^{-1}.
\end{equation}
\end{lemma}

\begin{proof}
Remember that $w_{ws}$ is strictly inside the good cone $C_g$, so it is strictly outside the bad cone $C_b$. This implies that there exists $n_1\geq n_0$ such that for all $n>n_1$ we have that $L^{-n}C_b$ is within $\frac {\theta_0}2$ distance from $E^{ss}$, which in turn implies that the angle between $\overline b$ and $L^{-n}C_b$ is greater than $\frac{\theta_0}2$.

We can assume that the $L^{-n}C_b$ is inside
\[
C_0:=\mathbb R\cdot\{\overline a+r\overline b:\ r\in(r_1,r_2)\},
\]
where $|r_{1,2}|<K_0$ for some $K_0>0$ (depends on $\theta_0$). 

Consider $t\geq t_0$. Since $Df_t^{-1}(L^{-n}(p))=Df_{-t}(L^{-n}(p))$ we get
\[
Df_t^{-1}(L^{-n}(p))(\overline a+r\overline b)=\overline a+(r-2\pi t\cos 2\pi x)\overline b.
\]
If $L^{-n}(p)\in G_t^{\alpha+}$ (or equivalently $(L^n\circ f_t)^{-1}(p)\in G_t^{\alpha+}$) then $\cos 2\pi x\in (t^{-\alpha},1]$ and since $r\in(r_1,r_2)$ it follows
\[
r-2\pi t\cos 2\pi x\in \left(r_1-2\pi t, r_2-2\pi t^{1-\alpha}\right).
\]
On the other hand, if $L^{-n}(p)\in G_t^{\alpha-}$ then $\cos 2\pi x\in [-1,-t^{-\alpha})$ and
\[
r-2\pi t\cos 2\pi x\in \left(r_1+2\pi t^{1-\alpha}, r_2+2\pi t\right).
\]
Let $v_+\in C^+$ and $v_-\in C^-$ be two unit vectors, 
\begin{align*}
v_+&=\pm\frac{\overline a +r_+\overline b}{\|\overline a +r_+\overline b\|}=\pm\frac{r_+^{-1}\overline a +\overline b}{\|r_+^{-1}\overline a +\overline b\|},\ \ r_+\in\left(r_1-2\pi t, r_2-2\pi t^{1-\alpha}\right),\\
v_-&=\pm\frac{\overline a +r_-\overline b}{\|\overline a +r_-\overline b\|}=\pm\frac{r_-^{-1}\overline a +\overline b}{\|r_-^{-1}\overline a +\overline b\|},\ \ r_-\in\left(r_1+2\pi t^{1-\alpha}, r_2+2\pi t\right).
\end{align*}
Then
\[
d_{\mathbb P}(v_+,v_-)\geq C_L\min\{|r_+-r_-|,|r_+^{-1}-r_-^{-1}|\}.
\]
Since
\[
|r_+-r_-|>4\pi t^{1-\alpha}-(r_2-r_1)>4\pi t^{1-\alpha}-2K_0
\]
and
\[
|r_+^{-1}-r_-^{-1}|>\frac 1{r_2+2\pi t}+\frac 1{2\pi t-r_1}>\frac 2{2\pi t+K_0}
\]
then choosing
\[
	t_1=\max\left\{t_0,\frac{K_0}{2\pi}+1,\left(\frac{K_0}{2\pi}\right)^{\frac 1{1-\alpha}}\right\}
\]
it follows that for any $t>t_1\geq t_0$, $d_{\mathbb P}(v_+,v_-)>C_Lt^{-1}$.

This gives the conclusion of the lemma.
\end{proof}

The expression $s_Lt^{-1}$ is the separation between $C^+$ and $C^-$. We have the following corollary.

\begin{corollary}
Suppose that $t>t_1$, $n>n_1$ and $\mathcal V$ is a family of unit vectors in $E^s$ such that $diam_{\mathbb P}(\mathcal V)<s_Lt^{-1}$. Then either $D(L^n\circ f_t)(p)v\in C_g$ for all $v\in\mathcal V$ and all $p\in G_t^{\alpha+}$, or $D(L^n\circ f_t)(p)v\in C_g$ for all $v\in\mathcal V$ and all $p\in G_t^{\alpha-}$.
\end{corollary}

\begin{proof}
$\mathcal V$ cannot intersect both $C^+$ and $C^-$ because of the bound on the diameter.
\end{proof}

\subsection{The Markov partition}

Fix $\mathcal P_0$ a Markov partition for $L$ (\cite{SinaiMarkov}); it is also a Markov partition for $L^n$. Then 
\[
\mathcal P=\{\text{connected component of }A\cap W^u(x)\text{ containing  }x:A\in\mathcal P_0, x\in\mathbb T^3\}
\]
is a partition subordinated to the unstable foliation $W^u$, that is, its atoms are inside unstable leaves; these are segments parallel to $v_u$, with the sizes uniformly bounded from above and below. We clearly have $L^n\mathcal P\prec L\mathcal P\prec\mathcal P$.

Since $L^n\circ f_t$ is homotopic to $L^n$, there is a continuous semiconjugacy $h:\mathbb T^3\rightarrow\mathbb T^3$ between $L^n\circ f_t$ and $L^n$, $h\circ L^n\circ f_t=L^n\circ h$. If $L^n\circ f_t$ is Anosov (condition \hyperlink{condA}{(A)} is satisfied) then $h$ is in fact a conjugacy. If $L^n\circ f_t$ is partially hyperbolic (condition \hyperlink{condPH}{(PH)} is satisfied), it is easy to see that $h$ takes homeomorphically strong unstable leaves of $L^n\circ f_t$ into unstable leaves of $L^n$ (it may take several strong unstable leaves to one unstable leaf if $h$ is not invertible).  Here is the argument: $h$ lifts to a proper semi-conjugacy  $\hat h:\mathbb R^3\to \mathbb R^3$ between a lift $F$ of $L^n\circ f_t$ and $L^n$, and thus $\hat h$ sends leaves of $\hat{W}^u_{t,n}$, the one dimensional unstable foliation of $F$, to unstable leaves of $L^n$. If $x,y\in \hat{h}^{-1}(p)$, then $\sup_{k\in\mathbb Z}d(F^kx,F^ky)<\infty$,  which by quasi-isometry of $\hat{W}^u_{t,n}$ \cite{DynCoh3Torus} prevents $x$ and $y$ to be on the same leaf. Therefore, for each $x\in\mathbb R^3$, $\hat h|:\hat{W}^u_{t,n}(x)\rightarrow \hat h(x)+E^u$ is continuous and injective, hence by invariance of the domain is an homeomorphism onto its image and $h(\hat{W}^u_{t,n}(x)) \subset \hat h(x)+E^u$ is an open interval, which again by quasi-isometry cannot have bounded diameter. This shows that $\hat h$ sends leaves of $\hat{W}^u_{t,n}$ to leaves of the unstable foliation of $L^n$, implying the corresponding analogous fact for $h$. See \cite{fisherpotriesamba} for more general discussion.  

For the rest of the article we assume that at least condition \hyperlink{condPH}{(PH)} is satisfied.

Let $\xi=h^{-1}(\mathcal P)\cap W^u_{t,n}$ be the partition subordinated to the unstable foliation of $L^n\circ f_t$, which projects to $\mathcal P$ under $h$ (this separates atoms in different unstable leaves of $W^u_{t,n}$). By eventually iterating $\mathcal P$ forward, we can assume that every atom of $\mathcal P$ is sufficiently large, and thus guarantee that every atom $P\in\xi$ intersects at least $20$ strips of $B_t^{\alpha}$ in the universal cover.

Let $\xi^{-1}=(L^n\circ f_t)^{-1}\xi$ be the preimage of $\xi$ under $L^n\circ f_t$. We denote by $\theta_u$ the angle between $E^u$ and the $yz$-plane. This angle is strictly greater than zero because the unstable line has irrational slope.

Observe that if $\gamma=\gamma_Lt\frac{\lambda_{ws}^n}{\lambda_u^n}$ is sufficiently small, then it is a good approximation to the angle of the cone $C^u_{\gamma}$, therefore if in addition $\gamma<\frac{\theta_u}2$ then the unstable foliation $W^u_{t,n}$ of $L^n\circ f_t$ is uniformly transverse to the foliation by $yz$-tori, with the angle larger that $\frac{\theta_u}2$. This will help us estimate the proportions of the unstable segments which are inside the good regions $G_t^{\alpha+}$, $G_t^{\alpha-}$ and the bad region $B_t^{\alpha}$.

We say that $S$ is a vertical strip in $\mathbb T^3$ of size $s$ if $S$ is the region between two $yz$-tori in $\mathbb T^3$, and the distance between the two $yz$-tori is $s$. We say that a $C^1$ curve $C$ in $\mathbb T^3$ is a $(S,\gamma)$-curve if $C$ is inside $S$, with the endpoints on the two $yz$-tori defining $S$, and with the angle between $TC$ and $E^u$ bounded from above by $\gamma$.

Let $\epsilon_M=\frac 1{20}$ (can be any small enough number). We will use the following lemma, since the proof is direct we will omit it.

\begin{lemma}
There exists $0<\gamma_M<\min\left\{\frac {\theta_u}2,\epsilon_M\right\}$ such that for any $\gamma<\gamma_M$, for any vertical strip $S$ of size $s$, $s\leq\frac 12$, and any $(S,\gamma)$-curve $C$, we have
\[
\left|l(C)-\frac s{\sin\theta_u}\right|<\epsilon_M\frac s{\sin\theta_u}.
\]
\end{lemma}

The lemma says that the length of an $(S,\gamma)$ curve can be $\epsilon_M$-approximated by the corresponding $E^u$ segment inside $S$, as long as the size of $S$ is smaller then one half.

\begin{definition}\hypertarget{condM}{}
If the following relation
\begin{equation}\label{eq:conditionM}
\gamma=\gamma_Lt\frac{\lambda_{ws}^n}{\lambda_u^n}<\gamma_M
\end{equation}
is satisfied, we say that we have the {\bf condition (M)}.
\end{definition}

Let us remark that the condition \hyperlink{condM}{(M)} is just a slight refinement of the condition \hyperlink{condPH}{(PH)}, and it says that the new unstable foliation $W^u_{t,n}$ of $L^n\circ f_t$ is close enough to the initial unstable foliation $E^u$ of $L$.

The useful properties of the partition $\xi$ are included in the next lemma. Recall that $\mu^{\xi}_P$ is the disintegration of the volume $\mu$ along the atom $P$ of the partition $\xi$, which is absolutely continuous with respect to Lebesgue on the unstable leaf and the densities satisfy \eqref{eq:density}.

\begin{lemma}\label{le:markov}
There exist $n_2\geq n_1$, $t_{\alpha}>0$ such that for any $n> n_2$ and $t^{\alpha}>t_{\alpha}$ if the condition \hyperlink{condM}{(M)} is satisfied then the partition $\xi$ defined above for $L^n\circ f_t$ satisfies the following properties.
\begin{enumerate}
\item
$\xi$ is subordinated to the strong unstable foliation $W^u_{t,n}$ of $L^n\circ f_t$. In particular $\xi^{-1}=(L^n\circ f_t)^{-1}(\xi)\succ \xi$, and the atoms of $\xi$ have the length uniformly bounded from above and below, meaning that there exist $d_L,D_L>0$ independent of $t$ and $n$ such that
$$
d_L<l(P)<D_L,\ \ \forall P\in\xi,
$$
where $l(P)$ is the length of the atom $P$.
\item
For any $P\in\xi$ we divide the ``pre-atoms'' from $\xi^{-1}$ inside $P$ into three subsets:
\begin{eqnarray*}
\mathcal B(P)&=&\{P'\in\xi^{-1}:\ P'\subset P,\ P'\cap B_t^{\alpha}\neq\emptyset\},\\
\mathcal G^+(P)&=&\{P'\in\xi^{-1}:\ P'\subset P\cap G_t^{\alpha+}\},\\
\mathcal G^-(P)&=&\{P'\in\xi^{-1}:\ P'\subset P\cap G_t^{\alpha-}\}.
\end{eqnarray*}
There exists $\delta_L>0$ such that we have the following bounds on the measures of the three subsets:
\begin{eqnarray*}
\mu_P^{\xi}\left(\cup_{P'\in\mathcal B(P)}P'\right)&<&\delta_L\left(t^{-\alpha}+\lambda_u^{-n}\right),\\
\mu_P^{\xi}\left(\cup_{P'\in\mathcal G^+(P)}P'\right)&>&\frac 13,\\
\mu_P^{\xi}\left(\cup_{P'\in\mathcal G^-(P)}P'\right)&>&\frac 13.\\
\end{eqnarray*}
\end{enumerate}
\end{lemma}

\begin{proof}
Part $(1)$ follows from the semiconjugacy $h$, once we choose the starting partition $\mathcal P$ for $L$. The bounds on the size of the atoms holds because the (semi)conjugacy $h$ preserves the (center)stable planes, and the new unstable bundle is close to the original one (cf. \cite{fisherpotriesamba}).

We now prove part $(2)$. Recall that $\xi$ is chosen so that every atom $P\in\xi$ intersects at least $20$ strips of $B_t^{\alpha}$ in the universal cover.

Since $(L^n\circ f_t)^{-1}$ contracts the unstable foliation by at least $\lambda_u^{-n}\frac{1-\gamma}{1+\gamma}$, we have that the size of every atom $P'\in\xi^{-1}$ is bounded from above by $D_L\lambda_u^{-n}\frac{1-\gamma}{1+\gamma}$. Let $B$ be the two vertical strips of the torus obtained by enlarging the strips of $B_t^{\alpha}$ by $D_L\lambda_u^{-n}\frac{1-\gamma}{1+\gamma}$ on each side. Then the complement of $B$ will consist of two vertical strips, $G^+\subset G_t^{\alpha+}$ and $G^-\subset G_t^{\alpha-}$. Since the size of the strips of $B_t^{\alpha}$ is bounded from above by $C_Lt^{-\alpha}$, we have that there exists $\delta_0>0$ such that the size $\delta$ of the strips of $B$ is bounded from above by $\delta_0\left(t^{-\alpha}+\lambda_u^{-n}\right)$. If $n$ and $t^{\alpha}$ are sufficiently large, we can assume that the size of the strips of $B$ satisfies
\[
\delta<\delta_0\left(t^{-\alpha}+\lambda_u^{-n}\right)<\frac{\epsilon_M}2.
\]
Then the sizes of the vertical strips $G^+$ and $G^-$ are bounded by below by $\frac{1-\epsilon_M}2$.

Let $P\in\xi$, and assume that $P$ intersects the strips of $B$ in the universal cover $k\geq 20$ times. Then we have the following:
\begin{itemize}
\item
$P$ is inside $k+1$ vertical strips of size $\frac 12$, and the length of each piece of $P$ in each strip is in $\left(\frac {1-\epsilon_M}{2\sin\theta_u},\frac {1+\epsilon_M}{2\sin\theta_u}\right)$. Also $P$ crosses completely at least $k-1$ such strips. Then the length of $P$ satisfies the bounds
\[
\frac{(k-1)(1-\epsilon_M)}{2\sin\theta_u}<l(P)<\frac{(k+1)(1+\epsilon_M)}{2\sin\theta_u}.
\]
\item
$\mathcal B(p)$ is contained in $k$ $(B_i,\gamma)$-curves, where $B_i,\ \ i=1,2$ are the 2 components of $B$. The the length of $\mathcal B(P)$ satisfies
\[
l(\mathcal B(P))<\frac{k\delta(1+\epsilon_M)}{\sin\theta_u}<\frac{k\delta_0(1+\epsilon_M)}{\sin\theta_u}(t^{-\alpha}+\lambda_u^{-n}).
\]
The relative length of $l(\mathcal B(P))$ in $P$ satisfies
\[
\frac{l(\mathcal B(P))}{l(P)}<\frac{2k\delta_0(1+\epsilon_M)}{(k-1)(1-\epsilon_M)}(t^{-\alpha}+\lambda_u^{-n})<\frac{840}{399}\delta_0(t^{-\alpha}+\lambda_u^{-n}).
\]
Remembering the bounds on the densities of $m_P$ with respect to Lebesgue in \ref{eq:density}, we get
\[
\mu_P^{\xi}(\mathcal B(P))<\frac{1+\gamma}{1-\gamma}\frac{840\delta_0}{399}(t^{-\alpha}+\lambda_u^{-n})<\frac{1+\gamma_M}{1-\gamma_M}\frac{840\delta_0}{399}(t^{-\alpha}+\lambda_u^{-n}).
\]
Let $\delta_L=\frac{1+\gamma_M}{1-\gamma_M}\frac{840\delta_0}{399}$.
\item
$\mathcal G^+(P)$ contains $\left[\frac{k-1}2\right]$ $(G^+,\gamma)$-curves, where $[x]$ is the integer part of $x$. Then the length of $\mathcal G^+(P)$ satisfies
\[
l(\mathcal G^+(P))>\left[\frac{k-1}2\right]\frac{(1-\epsilon_M)(1+\epsilon_M)}{2\sin\theta_u}>\frac{(k-3)(1-\epsilon_M^2)}{4\sin\theta_u}.
\]
The relative length of $\mathcal G^+(P)$ in $P$ satisfies
\[
\frac{l(\mathcal G^+(P))}{l(P)}>\frac{(k-3)(1-\epsilon_M^2)}{2(k+1)(1+\epsilon_M)},
\]
and the relative measure of $\mathcal G^+(P)$ in $P$ satisfies
\[
\mu_P^{\xi}(\mathcal G^+(P))>\frac 12\frac{k-3}{k+1}\frac{1-\epsilon_M^2}{1+\epsilon_M}\frac{1-\gamma_M}{1+\gamma_M}>\frac{1\cdot 17\cdot 19\cdot 19}{2\cdot 21\cdot 21\cdot 21}>\frac 13.
\]
\item The estimate of the relative measure of $\mathcal G^-(P)$ in $P$ is similar.
\end{itemize}
\end{proof}

\subsection{Preservation of Lipschitz vector fields}

Now we consider unit vector fields $X$ inside $E^s$, defined on the atoms of the partition $\xi$. Alternatively we can think of $X$ as a section in the projective bundle $\mathbb PE^s$ over the atoms of the partition $\xi$. Observe that $\mathbb PL^n|_{E^s}$ and $\mathbb PL^{-n}|_{E^s}$ have the Lipschitz constant $C_L\frac{\lambda_{ws}^n}{\lambda_{ss}^n}$. We also have that $\mathbb PDf_t(p)|_{E^s}$ and $\mathbb PDf_{-t}(p)|_{E^s}$ have the Lipschitz constant $C_Lt^2$ (as a function of both the unit vector $v$ and the base point on the torus $p$). The next step is to find a Lipschitz constant $l$, such that the $l$-Lipschitz vector fields are preserved by $\mathbb P(L^n\circ Df_t|_{E^s})$.

\begin{lemma}\label{le:lipschitz}
There exists a constant $l_L>0$ such that if $l=l_Lt^2\lambda_{ws}^{2n}<1$ then the image of $l$-Lipschitz unit vector fields in $\mathbb PE^s$ under the push-forward by the projectivization $\mathbb P(L^n\circ Df_t|_{E^s})$ are also $l$-Lipschitz unit vector fields.
\end{lemma}

\begin{proof}
Denote $f_{t,n}=L^n\circ f_t$. Let us evaluate the push forward of a $l$-Lipschitz unit vector field $X$ under $\mathbb PDf_{t,n}$:
\begin{eqnarray*}
d((\mathbb PDf_{t,n}X)(p),(\mathbb PDf_{t,n}X)(q))&=&d(\mathbb PDf_{t,n}(f_{t,n}^{-1}(p))X(f_{t,n}^{-1}(p)),\mathbb PDf_{t,n}(f_{t,n}^{-1}(q))X(f_{t,n}^{-1}(q)))\\
&\leq&d(\mathbb PDf_{t,n}(f_{t,n}^{-1}(p))X(f_{t,n}^{-1}(p)),\mathbb PDf_{t,n}(f_{t,n}^{-1}(p))X(f_{t,n}^{-1}(q)))+\\
& &\ \ +d(\mathbb PDf_{t,n}(f_{t,n}^{-1}(p))X(f_{t,n}^{-1}(q)),\mathbb PDf_{t,n}(f_{t,n}^{-1}(q))X(f_{t,n}^{-1}(q))).
\end{eqnarray*}

Let $p_n=f_{t,n}^{-1}(p)$, $q_n=f_{t,n}^{-1}(q)$. Observe that $d(p_n,q_n)\leq C_L\lambda_u^{-n}$. We continue estimating:
\begin{eqnarray*}
d(\mathbb PDf_{t,n}(p_n)X(p_n),\mathbb PDf_{t,n}(p_n)X(q_n))&=& d(\mathbb PL^n\left(\mathbb PDf_t(p_n)X(p_n)\right), \mathbb PL^n\left(\mathbb PDf_t(p_n)X(q_n)\right)\\
&\leq&C_L\frac{\lambda_{ws}^n}{\lambda_{ss}^n}d(\mathbb PDf_t(p_n)X(p_n),\mathbb PDf_t(p_n)X(q_n))\\
&\leq&C_Lt^2\frac{\lambda_{ws}^n}{\lambda_{ss}^n}d(X(p_n),X(q_n))\\
&\leq&C_Lt^2\frac{\lambda_{ws}^n}{\lambda_{ss}^n\lambda_u^n}ld(p,q).
\end{eqnarray*}

On the other hand
\begin{eqnarray*}
d(\mathbb PDf_{t,n}(p_n)X(q_n),\mathbb PDf_{t,n}(q_n)X(q_n))&=& d(\mathbb PL^n\left(\mathbb PDf_t(p_n)X(q_n)\right), \mathbb PL^n\left(\mathbb PDf_t(q_n)X(q_n)\right)\\
&\leq&C_L\frac{\lambda_{ws}^n}{\lambda_{ss}^n}d(\mathbb PDf_t(p_n)X(q_n),\mathbb PDf_t(q_n)X(q_n))\\
&\leq&C_Lt^2\frac{\lambda_{ws}^n}{\lambda_{ss}^n}d(p_n,q_n)\\
&\leq&C_Lt^2\frac{\lambda_{ws}^n}{\lambda_{ss}^n\lambda_u^n}d(p,q).
\end{eqnarray*}

Since the product of the three eigenvalues is one we have that $\lambda_{ss}\lambda_u=\lambda_{ws}^{-1}$. Putting together the two estimates from above we obtain that if $X$ is a $l$-Lipschitz unit vector field on the element $P$ of $\xi$, then $\mathbb P(Df_t\circ L^n)X$ is $C_Lt^2\lambda_{ws}^{2n}(l+1)$-Lipschitz on $f_t\circ L^n(P)$.

We can now take $l_L=2C_L$ and then for $l=l_Lt^2\lambda_{ws}^{2n}<1$ we can easily check that the conclusion holds.
\end{proof}

\begin{definition}\hypertarget{condL}{}
If the relation $l_Lt^2\lambda_{ws}^{2n}<1$ is satisfied then we say that we have the {\bf Condition (L)}.
\end{definition}

\begin{remark}
If the vector field $X$ on $P\in\xi$ is Lipschitz with constant $l=l_Lt^2\lambda_{ws}^{2n}$ then, since the diameter of $P$ is bounded by $D_L$, we get that $\hbox{var}(X)< l_LD_Lt^2\lambda_{ws}^{2n}$.
\end{remark}

\section{Proof of Theorem A}

Now we are ready for the proof of \hyperlink{theoremA}{Theorem A}. 

\begin{proof}[Proof of \hyperlink{theoremA}{Theorem A}]

It is more convenient to work with the inverse of $L$, which will have two stable exponents, and in this case we have to show that $L^n\circ f_t$ decreases the strong stable exponent. Since the perturbations keeps the sum of the two stable exponents constant, this is equivalent to proving that the weak stable exponent increases.

Let us assume that $n>n_2(\geq n_1\geq n_0)$, $t>t_1(\geq t_0)$ and $t^{\alpha}>t_{\alpha}$.

We assume also that the conditions \hyperlink{condA}{(A)}, \hyperlink{condM}{(M)} and \hyperlink{condM}{(L)} are satisfied, so it holds
\begin{align*}
&t\lambda_{ws}^n<\frac 1{a_L},\\
&t\frac{\lambda_{ws}^n}{\lambda_u^n}<\frac{\gamma_M}{\gamma_L} \shortintertext{and}
&t^2\lambda_{ws}^{2n}<\frac 1{l_L}.
\end{align*}

We also assume that the variation of the $l$-Lipschitz unit vectorfields on the atoms of the partition $\xi$ is smaller that the separation of the preimages of the bad cone:
\begin{align*}\hypertarget{condSL}{}
&l_LD_Lt^2\lambda_{ws}^{2n}<s_Lt^{-1}, \shortintertext{or}
&t^3\lambda_{ws}^{2n}<\frac{s_L}{l_LD_L}\hspace{2cm} \text{(we call this {\bf Condition (SL)})}
\end{align*}

Let us remark that all the conditions above are satisfied if we take
\begin{equation}\label{eq:tn}
t=\lambda_{ws}^{-n\nu}
\end{equation}
for some $\nu\in\left(0,\frac 23\right)$ and for $n$ sufficiently large.

From Lemma \ref{le:markov} we have that
\begin{eqnarray*}
\mu_P^{\xi}\left(\cup_{P'\in\mathcal B(P)}P'\right)&<&\delta_L\left(t^{-\alpha}+\lambda_u^{-n}\right)\\
\mu_P^{\xi}\left(\cup_{P'\in\mathcal G^+(P)}P'\right)&>&\frac 13\\
\mu_P^{\xi}\left(\cup_{P'\in\mathcal G^-(P)}P'\right)&>&\frac 13.
\end{eqnarray*}
Since $t^{-\alpha}=\lambda_{ws}^{n\alpha\nu}>\lambda_u^{-n}$, the first relation can be rewriten as
\[
\mu_P^{\xi}\left(\cup_{P'\in\mathcal B(P)}P'\right)< 2\delta_Lt^{-\alpha}:=\delta.
\]
Let $\mathcal{X}$ be the family of unit vectorfields defined on atoms of $\xi$ and which are $l$-Lipschitz on the atom, $l=l_Lt^2\lambda_{ws}^{2n}<1$:
\[
\mathcal{X}=\{ X:P\rightarrow \mathbb PE^s:\ \ P\in\xi,\ X\hbox{ is $l$-Lipschitz}\}.
\]
Lemma \ref{le:lipschitz} shows that the family $\mathcal X$ is invariant under the pushed forward of the projectivization of $L^n\circ f_t$ (in the sense that the image of a $l$-Lipschitz unit vectorfield on $P\in\xi$ is several $l$-Lipschitz unit vectorfields defined on the atoms of $\xi$ from $(L^n\circ f_t)(P)$).

We consider the derivative cocycle of $L^n\circ f_t$ restricted to $E^s$, $D(L^n\circ f_t)|_{E^s}$.

We will check that the family of vector fields $\mathcal X$ is adapted as in Definition \ref{def:adapted}, for $\delta=2\delta_Lt^{-\alpha}$ and $\lambda=C_L\lambda_{ws}^{n-n\nu(1-\alpha)}$. We divide the vectorfields into ``good'' and ``bad'' in the following way:
\begin{eqnarray*}
\mathcal X^g&=&\{ X\in\mathcal X:\ X(p)\in C_g,\ \forall p\in P, \hbox{ where $P$ is the domain of $X$}\}\\
\mathcal X^b&=&\{ X\in\mathcal X:\ \exists p\in P \hbox{ with } X(p)\notin C_g, \hbox{ where $P$ is the domain of $X$}\}
\end{eqnarray*}
The good region is $G=G_t^{\alpha}$.
\begin{itemize}
\item
The condition \ref{item.condicion1} follows from the construction of the Markov partition in Lemma \ref{le:markov}.
\item
The condition \ref{item.condicion2} is obvious, we can take $X$ to be a constant vector in the good cone $C_g$.
\item
The condition \ref{item.condicion3} follows from Proposition \ref{prop:expansion}. The lower bound in \ref{eq:strong} is $C_L\lambda_{ws}^nt^{1-\alpha}$, but since we chose $t=\lambda_{ws}^{-n\nu}$ we have in fact the lower bound $\lambda=C_L\lambda_{ws}^{n-n\nu(1-\alpha)}$.
\item
The condition \ref{item.condicion4} follows also directly from Proposition \ref{prop:expansion}.
\item
The condition \ref{item.condicion5} is the most delicate. It follows from our choice of the Lipschitz constant of the vector fields in $\mathcal X$. The condition \hyperlink{condSL}{(SL)} implies that the variation of the vector field $X$ on an atom $P$ of the partition is smaller that the separation between the preimages of the bad cone given by Lemma \ref{le:separation}. This implies that the push forward of the vector field has to go inside the good cone at least on one of the sets $\mathcal G^+(P)$ or $\mathcal G^-(P)$. The choice of the Markov partition from Lemma \ref{le:markov} will give us the required condition.
\end{itemize}

Now we are in condition to apply Theorem \ref{teo.tecnico}. The cocycle is
\[
A=D(L^n\circ f_t)|_{E^s}=L^n\cdot Df_t|_{E^s}.
\]
The lower bound for the expansion in the good region is
\[
\lambda=C_L\lambda_{ws}^nt^{1-\alpha}=C_L\lambda_{ws}^{n-n\nu(1-\alpha)}.
\]
From Proposition \ref{prop:expansion} we see that the norm of the inverse of the cocycle is bounded by
\[
\| A^{-1}\|<C_L\lambda_{ss}^{-n}t=C_L\lambda_{ss}^{-n}\lambda_{ws}^{-n\nu}.
\]
Remember that also
\[
\delta=2\delta_Lt^{-\alpha}=2\delta_L\lambda_{ws}^{n\nu\alpha}.
\]
Since we have an adapted family of vector fields (for $\beta=\frac 13$), we obtain the inequality
\begin{eqnarray*}
\chi^+&>&\frac 1{1+3\delta}\log\frac {\lambda^{1-\delta}}{\| A^{-1}\|^{4\delta}}\\
&\geq&-C_L+\frac 1{1+3\delta}\log\frac{\lambda_{ws}^{[n-n\nu(1-\alpha)](1-\delta)}}{\lambda_{ss}^{-4n\delta}\lambda_{ws}^{-4n\nu\delta}}\\
&=&-C_L+\frac 1{1+3\delta}\log\left(\lambda_{ws}^{n-n\nu(1-\alpha)+n\delta(5\nu-\nu\alpha-1)}\lambda_{ss}^{4n\delta}\right)\\
&=&-C_L+\log\lambda_{ws}^n-\nu(1-\alpha)\log\lambda_{ws}^n+n\delta\left(\frac {-3\log\lambda_{ws}^{1-\nu(1-\alpha)}+\log\lambda_{ws}^{5\nu-\nu\alpha-1}+\log\lambda_{ss}^4}{1+3\delta}\right)\\
&=&-C_L+\log\lambda_{ws}^n-\nu(1-\alpha)\log\lambda_{ws}^n+\frac{n\delta C_{L,\alpha,\nu}}{1+3\delta}.
\end{eqnarray*}

Paring this with the facts $\delta=2\delta_L\lambda_{ws}^{n\nu\alpha}$ and $\lambda_{ws}<1$, we deduce
\[
\lim_{n\rightarrow\infty}\frac{n\delta C_{L,\alpha,\nu}}{1+3\delta}=\lim_{n\rightarrow\infty}\frac{2n\delta_L\lambda_{ws}^{n\nu\alpha} C_{L,\alpha,\nu}}{1+6\delta_L\lambda_{ws}^{n\nu\alpha}}=0,
\]
so the last term is negligible for large $n$. Then for such $n$'s we will obtain that the weak stable Lyapunov exponent of $L^n\circ f_t$ is greater than the weak stable Lyapunov exponent of $L^n$ ($=\log\lambda_{ws}^n$) by at least $-\nu(1-\alpha)\log_{ws}^n-C_L$, which is positive.

This finishes the proof.
\end{proof}

\begin{remark}
Let us comment that the maximal change that we can achieve for the weak stable Lyapunov exponent is close to two thirds of the weak stable exponent of the linear map $L^n$. This is done for $n$ large, $\nu$ close to $\frac 23$ and $\alpha$ close to $0$.

Then, in the context of \hyperlink{theoremA}{Theorem A}, we can increase the top Lyapunov exponent (strong unstable) by at most two thirds of the size of the second Lyapunov exponent (weak unstable).
\end{remark}

\section{Proof of Theorem B}

In this section we will prove \hyperlink{theoremB}{Theorem B}. First we explain how one can addapt the considerations from Section \ref{sec:prelim}, which were made for $\mathbb T^3$, to the four-dimensional torus.

We will take the maps $f_t$ preserving the planes parallel to $E^s_1:=E^{ws}\oplus E^{ms}$. We can assume again that this plane is close to being horizontal, and is generated by the vectors $\overline a_1=(1,0,a_1,a_2)$ and $\overline b_1=(0,1,b_1,b_2)$. Then the map $f_t$ will preserve the lines parallel to $\overline b$, and on each such line is a translation by $t\sin(2\pi x)\overline b$.

The bound on the derivatives $Df_t$ and $DF_t^{-1}$ is still $C_Lt$, while the Lipschitz constant of the projectivization of the derivative restricted to the $E^s_1$ plane also stay the same, $C_Lt^2$. The condition \hyperlink{condPH}{(PH)} stays the same, $\gamma_Lt\frac{\lambda_{ws}^n}{\lambda_u^n}<1$, and it implies that $L^n\circ f_t$ is partially hyperbolic, and it also gives the same bounds on the angle between $E^u_{t,n}$ and $E^u$, and on the densities of the desintegration of the volume along unstable leaves.

If we want that $L^n\circ f_t$ is partially hyperbolic with a splitting into three subbundles (i.e.\@ there is also a strong stable bundle), then we also need to add the {\bf condition (PH')}:\hypertarget{condPHp}
\[
\gamma_L't\frac{\lambda_{ss}^n}{\lambda_{ms}^n}<1,
\]
which will give a strong stable cone invariant be the inverse of $L^n\circ f_t$.

The formula \eqref{eq:strong} stays the same, but the formula \eqref{eq:weak} from Proposition \ref{prop:expansion} changes, $\lambda_{ss}$ is replaced by $\lambda_{ms}$ because it is the stronger contraction in $E^s_1$:

\begin{equation}\label{eq:weak2}
\frac{\|D(L^n\circ f_t)(p)v\|}{\|v\|}> C_L\frac {\lambda_{ms}^n}t.
\end{equation}

There is no change in the separation of the preimages of the bad cone, nor in the construction of the Markov partition.

There is however a change in the Lipschitz constant of the unit vectorfields preserved by the projectivization of $L^n\circ f_t$; $l=l_Lt^2\lambda_{ws}^n<1$ from Lemmma \ref{le:lipschitz} becomes\hypertarget{condLp}{}
\[
l=l_Lt^2\lambda_{ws}^{2n}\lambda_{ss}^n<1,
\]
and this is the new {\bf condition (L')}, which also gives the new Lipschitz constant.

This will also change the condition (SL) to the {\bf condition (SL')}:\hypertarget{condCLp}{}
\[
t^3\lambda_{ws}^{2n}\lambda_{ss}^n<\frac{s_L}{l_Ld_L}.
\]

Now we are ready for the proof of the Theorem.

\begin{proof}[Proof of \hyperlink{theoremB}{Theorem B}]

Assume that we have the conditions \hyperlink{condPH}{(PH)}, \hyperlink{condM}{(M)}, \hyperlink{condLp}{(L')} and \hyperlink{condCLp}{(SL')}:
\begin{eqnarray*}
t\frac{\lambda_{ws}^n}{\lambda_u^n}&<&\min\left\{\frac 1{\gamma_L},\frac{\gamma_M}{\gamma_L}\right\},\\
t^2\lambda_{ws}^{2n}\lambda_{ss}^n&<&\frac 1{l_L}\\
t^3\lambda_{ws}^{2n}\lambda_{ss}^n&<&\frac{s_L}{l_Ld_L}.
\end{eqnarray*}

We also assume that $t$ and $n$ are sufficiently large so all the previous results apply. Again, all the conditions above are satisfied if $t=\lambda_{ws}^{-n(1+\nu)}$, where $n$ is sufficiently large and $\nu>0$ satisfies
\[
\nu<\min\left\{ \frac{\log\lambda_u}{-\log\lambda_{ws}},\frac 13\left(\frac{-\log\lambda_{ss}}{-\log\lambda_{ws}}-1\right)\right\}.
\]
There exists such $\nu$ because of the condition $1>\lambda_{ws}>\lambda_{ss}$.

The condition \hyperlink{condPHp}{(PH')} means that
\[
t\frac{\lambda_{ss}^n}{\lambda_{ms}^n}<\frac 1{\gamma_L'},
\]
and this would give an extra restriction, $\nu<\frac{\log\lambda_{ss}-\log\lambda_{ms}}{\log\lambda_{ws}}-1$, and this can only happen if we have the extra condition on the stable exponents of $L$, $\lambda_{ss}<\lambda_{ws}\lambda_{ms}$.

The Markov partition $\xi$ and the family of Lipschitz unit vectorfields $\mathcal X$ are constructed as before. The partition into good and bad vectorfields is similar, and a similar argument shows that the family of vector fields is adapted to the derivative cocycle restricted to $E^s_1$.

Then we can apply again the Theorem \ref{teo.tecnico} for $A=D(L^n\circ f_t)|_{E^s_1}$, the Markov partition $\xi$, the adapted family $\mathcal X$. In this case we will have
\[
\| A^{-1}\|=C_L\lambda_{ms}^{-n}t=C_L\lambda_{ws}^{-n(1+\nu)}\lambda_{ms}^{-n}
\]
\[
\lambda=C_L\lambda_{ws}^nt^{1-\alpha}=C_L\lambda_{ws}^{n-n(1+\nu)(1-\alpha)},
\]
and
\[
\delta=2\delta_Lt^{-\alpha}=2\delta_L\lambda_{ws}^{n\nu\alpha}.
\]

We do not know if the map $L^n\circ f_t$ is ergodic, however due to Remark \ref{rem.nonergodic} (the Markov partition $\xi$ is clearly generating) we obtain the bound on the Lyapunov exponent at almost every point with respect to the volume:
\begin{eqnarray*}
\chi^+&>&\frac 1{1+3\delta}\log\frac {\lambda^{1-\delta}}{\| A^{-1}\|^{4\delta}}\\
&\geq&-C_L+\frac 1{1+3\delta}\log\frac{\lambda_{ws}^{[n-n(1+\nu)(1-\alpha)](1-\delta)}}{\lambda_{ms}^{-4n\delta}\lambda_{ws}^{-4n(1+\nu)\delta}}\\
&=&-C_L+\frac 1{1+3\delta}\log\lambda_{ws}^{-n(\nu-\alpha-\nu\alpha)+n\delta(5\nu+4-\alpha-\nu\alpha)}\lambda_{ms}^{4n\delta}\\
&=&-C_L-n(\nu-\alpha-\nu\alpha)\log\lambda_{ws}+n\delta\left(\frac {3\log\lambda_{ws}^{\nu-\alpha-\nu\alpha}+\log\lambda_{ws}^{5\nu+4-\alpha-\nu\alpha}+\log\lambda_{ms}^4}{1+3\delta}\right)\\
&=&-C_L-n(\nu-\alpha-\nu\alpha)\log\lambda_{ws}+\frac{n\delta C_{L,\alpha,\nu}}{1+3\delta}.
\end{eqnarray*}

Again for large $n$ the last term is negligible, and if we choose $\nu$ and $\alpha$ such that $\nu-\alpha-\nu\alpha>0$ (for example $\alpha=\frac {\nu}2<\frac 12$) then for large $n$ we will get that the exponent is positive.

Since the strong unstable exponent exponent is unchanged, making another exponent positive will automatically increase the sum of the positive exponents (by Pesin's formula \cite{LyaPesin}), and thus the metric entropy with respect to the volume. This finishes the proof.

\end{proof}

\section{Appendix A: Construction of Adapted families} 
\label{sec:appendix}

In this appendix we will demostrate that if the map is expanding on $\xi$ and has a (non-necessarily invariant) expanding cone for the cocycle, then one can construct an adapted family. The method works for H\" older cocycles, but one needs to assume that the expansion is sufficiently strong in order to compensate the H\" older constant and the bolicity\footnote{Recall that the bolicity of an invertible matrix $A$ is $\norm{A}\cdot\norm{A^{-1}}$.} of the cocycle. We also take a different approach, instead of considering the separation of the preimages of the good cone, we cassume that the iterates of the bad vectors from some specific regions go stricly inside the good cones (with a separation from the boundary of the good cone). It is useful to keep in mind that in our construction of $L^n\circ f_t$ we have $s=2$ and the good regions $G_t^{\alpha+}$ and $G_t^{\alpha-}$ are the regions which cointain the atoms $\xi^{c,1}$ and $\xi^{c,2}$.

The standing hypothesis for this Appendix are the same as in Section \ref{sec.detecting}. We assume that the hypothesis \ref{item.condicion1} is satisfied. There exists the good region $G\subset M$ such that for each $\p_0\in\xi$ we have
   		\[
		\mc\left(\cup\left\{\p\in\xi^1|\p_0:\ \p\subset G\right\}\right)>1- \delta.
   		\]

Besides hypothesis \ref{item.condicion1}, we assume the following.

\begin{enumerate}

\item[\referencia{item.condicion6}{\rm H6}] The map $f$ is expanding in the partition $\xi$, meaning that there exists $\mathtt{d}>1$ such that for every $\p\in \xi, m,m'\in \p$ it holds $d_M(fm,fm')\geq \mathtt{d}\cdot d_M(m,m')$.

\item[\referencia{item.condicion7}{\rm H7}] There exists a continuous sector $\Delta=\{\Delta\}_{m\in M}$ in $\mathbb{P}E$, a finite partition of the complementary sector $\Delta_m^c=\{\mathbb{P}E_m\setminus\Delta_m\}_{m\in M}=\Delta^{c,1},\ldots,\Delta^{c,s}$, a corresponding family of atoms $\xi^c=\xi^{c,1}\cup\cdots\xi^{c,s}\subset \xi^1$  and constants $\lambda,\alpha>0, k>s$ satisfying: 
	\begin{enumerate}
	
		\item If $m\in \p$ then
		\begin{itemize}
		    \item $\PA(m)(\Delta_m)\in \Delta_{fm}$ and furthermore $\disH(\PA(m)(\Delta_m),\partial \Delta_{fm})\geq \alpha$;
	 		\item if $[v]\in \Delta_m$, then $\norm{A(m)\cdot v}\geq \lambda \norm{v}$.
		\end{itemize}

		\item If $\p\in\xi$ then
		\begin{itemize}
			\item for every $\p'\in \xi^{c,i}|\p$ it holds that $m'\in \p'$ implies $\PA(\Delta^{c,i}_{m'})\subset \Delta_{fm'}$, and furthermore $\disH(\PA(\Delta^{c,i}_{m'}), \partial\Delta_{fm'})\geq \alpha$;
	 		\item for every $1\leq i\leq s$,
	 		\[
				\mc\left(\{\p'\in \xi^{c,i}|\p\}\right)\geq \frac{1}{k}.
   			\]
   		\end{itemize}
	
	\smallskip

	\end{enumerate}
	
	\item[\referencia{item.condicion8}{\rm H8}] $\pi:E\to M$ is H\"older, with the H\"older exponent $\theta$. $(f,A)$ is a $\theta$-H\"older cocycle, so $(f,\PA)$ is a $(C_{\PA},\theta)$ projective cocycle for some $C_{\PA}>0$. Let $\mathtt{q}=\frac{b(A)}{\mathtt{d}^{\theta}}, \mathtt{r}:=\sup_{m\in M}\{\mathrm{diam}(\p):\p\in\xi\}$. It holds 
		\begin{itemize}
		\item $\displaystyle{\mathtt{q}<1}$.
		\item $\displaystyle{\frac{C_{\PA}^2\mathtt{r}^{\theta}}{(1-q)\mathtt{d}^{\theta}}<\frac{\alpha}{2}}$.
		\end{itemize}
\end{enumerate}

Consider the constant

\begin{equation}
C_0:=\frac{C_{\PA}}{(1-\mathtt{q})\mathtt{d}^{\theta}}.
\end{equation}

and define

\begin{equation}
\mathcal{X}:=\{X:\p_0\to \mathbb{P}E: X\text{ is } (C_X,\theta)-\text{H\"older with } C_X\leq C_0\}.
\end{equation}

\begin{lemma}\label{lem.adapted}
$\mathcal{X}$ is invariant under $(f,A)$. 
\end{lemma}

\begin{proof}
Take $X:\p_0\to E$ and consider $Y^1_{\p}$. Note that $\forall m,m'\in \p_0$,
\begin{align*}
&\disang(\PA(m)\cdot X(m),\PA(m)\cdot X(m'))\leq b(A) \disang(X(m),X(m'))\leq b(A) C_X\cdot d_M(m,m')^{\theta}\\
&\disang(\PA(m)\cdot X(m'),\PA(m')\cdot X(m'))\leq C_{\PA}d_M(m,m')^{\theta}
\end{align*}
hence 
\begin{equation*}
\disang(\PA(m)\cdot X(m),\PA(m')\cdot X(m'))\leq (b(A)C_X+C_{\PA})d_M(m,m')^{\theta}.
\end{equation*}
This in turn implies
\begin{align*}
	\disang(Y^1_{\p}(m),Y^1_{\p}(m'))&\leq (b(A)C_X+C_{\PA})d_M(g_{\p}m,g_{\p}m')^{\theta}\leq \frac{b(A)C_X+C_{\PA}}{\mathtt{d}^{\theta}}d_M(m,m')^{\theta}\\
	&\leq C_0 d_M(m,m')^{\theta}
\end{align*}
by \ref{item.condicion8}. By a direct induction argument it follows that $\mathcal{X}$ is invariant.
\end{proof}

Now define
\begin{equation}
\mathcal{X}^g:=\{X:\p_0\to \mathbb{P}E\in\mathcal{X}: \forall m\in\p_0, X(m)\in\Delta_m\},\quad \mathcal{X}^b=\mathcal{X}\setminus\mathcal{X}^g.
\end{equation}

It is simple to check that with this decomposition $\mathcal{X}$ satisfies conditions \ref{item.condicion1}, \ref{item.condicion2} and \ref{item.condicion3}. The remaining ones are considered in the following lemma.

\begin{lemma}\label{lem.transiciones}	
Consider $X:\p_0\to E\in\mathcal X$ and $\p\in \xi^1|\p_0$.
\begin{enumerate}
\item If $X\in\mathcal{X}^g$ and $\p\in \xi^1|_{\p_0}, \p\in G$ then $Y^1_{\p}\in\mathcal{X}^g$.
\item If $X\in\mathcal{X}^b$ there exists $\p_0^r(X)\subset \xi^1|\p_0$ satisfying:
	\begin{itemize}
	\item $\mc\Big(\bigcup_{\p\in\p_0^r(X)}\p\Big)\geq \frac 1k$.
	\item For every $\p\in \p_0^r(X)$, $Y^1_{\p}\in\mathcal X^g$.
	\end{itemize}
\end{enumerate}	
\end{lemma}

\begin{proof}
The first part is direct consequence from \ref{item.condicion7}, $(a)$ . As for the second, by our assumptions there exists $m_0\in \p_0$ such that $X(m_0)\in \Delta^{c,i}_{m_0}$, for some $1\leq i\leq s$. Consider the atoms $\p\in \xi^{i,c}|\p_0$, and observe that by the requirements assumed in \ref{item.condicion7} $(b)$ one gets that $m\in \p$ implies $\PA(m)\cdot X(m_0)\in \Delta_{fm}$ and 
\[
	\disH(\PA(m)\cdot X(m_0) ,\partial \Delta_{fm})\geq \alpha.
\]
On the other hand,
\[
	\disH(\PA(m)\cdot X(m_0),\PA(m)\cdot X(m))\leq C_{\PA}\cdot C_0\mathrm{r}^{\theta}
\]
which in turn implies
\[
	\disH(\PA(m)\cdot X(m) ,\partial \Delta_{fm})\geq \alpha-C_{\PA}\cdot C_0\mathrm{r}^{\theta}\geq\frac{\alpha}{2}
\]
by \ref{item.condicion8}. We conclude that for $m'\in f(\p)$ we have $Y^1(m')\in \Delta(m')$; this implies the second statement of the Lemma.
\end{proof}

From Theorem \ref{teo.tecnico} we deduce: 

\begin{corollary}
If \ref{item.condicion6}-\ref{item.condicion8} hold then $\mathcal{X}$ is an adapted family for the cocycle $(f,A)$, and thus the top Lyapunov exponent satisfies
$$
\chi^+>\frac 1{1+k\delta}\log\frac{\lambda^{1-\delta}}{\| A^{-1}\|^{(k+1)\delta}}.
$$
\end{corollary}
\section{Appendix B: Robustness of the examples} 
\label{sec:appendixb}

The constructions are robust in the $\mathcal{C}^2$ topology, meaning that for a small $\mathcal{C}^2$ perturbation of $L^n\circ f_t$ one will still have similar bounds for the Lyapunov exponents. This allows for example to realize more Lyapunov exponents for the Anosov maps homotopic to the Anosov automorphism, because generic small perturbations will decrease the unstable exponent and will increase the sum of the two stable exponents. Observe that the construction is not robust in the $\mathcal{C}^1$ topology due to the results of Ma\~n\'e-Bochi-Viana (cf.\@ \cite{Bochi2005}).

We will see that this robustness works even if we go out of the volume preserving world. In change we have to consider the Lyapunov exponents with respect to SRB measures.

\begin{proposition}\label{prop.robust}
Suppose that $L^n\circ f_t$ is the map constructed in the proof of \hyperlink{theoremA}{Theorem A}. Then there exists a $\mathcal{C}^2$ neighborhood $\mathcal U$ of $L^n\circ f_t$ such that for every $f\in\mathcal U$, the bottom Lyapunov exponent of $f$ with respect to the SRB measure of $f$ will verify the same bounds as the bottom Lyapunov exponent of $L^n\circ f_t$ with respect to the volume.
\end{proposition}

A similar result could be obtained for the map constructed in the example from the proof of \hyperlink{theoremB}{Theorem B}, however it would work for the strong partially hyperbolic case, and with an additional bunching condition so that the centre bundle is H\H older with exponent close to 1. This could allow to create an ergodic example, if the starting map $L$ is sufficiently center bunched.

\begin{proof}
We will explain how to prove the Proposition \ref{prop.robust}. 

Let $f$ be a $\mathcal{C}^2$ map $\mathcal{C}^2$ close to $L^n\circ f_t$. Then $f$ is Anosov, the stable and unstable manifolds are $\mathcal{C}^2$, and the stable and unstable foliations ($W^s_f$ and $W^u_f$) are $\mathcal{C}^1$. The hyperbolic splitting $TM=E^s_f\oplus E^u_f$ is also $\mathcal{C}^1$ and is close to the hyperbolic splitting for $L^n\circ f_t$. Also $f$ and $L^n\circ f_t$ are conjugated, and the conjugacy is $\mathcal{C}^0$ close to the identity.

The unstable foliation is absolutely continuous, and the disintegrations of the SRB measure of $f$ along the unstable leaves of $W^u_f$ will have densities verifying the same formula:
\[
\frac{\rho_f(x)}{\rho_f(y)}=\lim_{k\rightarrow\infty}\frac{\left|Df^{-k}|_{E^u_{f}(x)}\right|}{\left|Df^{-k}|_{E^u_f(y)}\right|},
\]
for $x$ and $y$ in the same unstable leaf $W^u_f$. This implies that for $f$ $C^2$ close to $L^n\circ f_t$ we have the good bounds on the variation of densities on local unstable manifolds, in other words \ref{eq:density} holds (here is where the $\mathcal{C}^2$ condition is important).

The good regions $G_t^{\alpha+}$ and $G_t^{\alpha-}$ and the bad region $B_t^{\alpha}$ are the same. The new good cone $C_g^f$ is the orthogonal projection of $C_g$ to the new stable space $E^s_f$, and it is $\mathcal{C}^0$ close to $C_g$. We consider the cocycle $A_f=Df|_{E^s_f}$, and this is $\mathcal{C}^1$ and close to the cocycle $A=D(L^n\circ f_t)|_{E^s}$. We can assume that we have the same upper bound on $\|A\|$ and $\|A^{-1}\|$, and the bounds \ref{eq:strong} and \ref{eq:weak} from Proposition \ref{prop:expansion} will hold for $f$ (with the good cone $C_g$ replaced by $C_g^f$).

Since both $C_g^f$ (or $C_b^f=(C_g^f)^c$) and $A_f$ depend continuously of $f$ in the $\mathcal{C}^1$ topology, the same separation of the preimages of the bad cone from Lemma \ref{le:separation} will hold for $f$ sufficiently close to $L^n\circ f_t$.

Let $h_f$ be the conjugacy between $f$ and $L^n\circ f_t$. We will consider the unstable Markov partition $\xi_f=h_f^{-1}(\xi)$, which is in fact $\mathcal{C}^0$ close to $\xi$. Again the continuity of the partition $\xi_f$ with respect to $f$, and the bounds on the densities of the disintegrations of the SRB measure along unstable manifolds will imply that the Lemma \ref{le:markov} will also hold for $f$ (with $\xi$ replaced by $\xi_f$ and the volume replaced by the SRB measure of $f$).

The projectivization $\mathbb PA_f$ is also $\mathcal{C}^1$ and close to $\mathbb PA$. Again we can assume that we have the same bounds for the Lipschitz constant of $\mathbb P(A)$ as $\mathbb P(L^n\circ f_t)$ (here we use that $f$ is $\mathcal{C}^2$ close to $L^n\circ f_t$ and the bundle $E^s_f$ is Lipschitz with Lipschitz constant close to 1). This will give us the preservation of $l$-Lipschitz vector fields along the unstable leaves for the same $l$ as in Lemma \ref{le:lipschitz} for $L^n\circ f_t$.

Then we can show again that the family of $l$-Lipschitz vector fields on the atoms of $\xi_f$ is addapted for the same $\delta$, $\lambda$ and the bound on $\|A_f^{-1}\|$, and the conclusion follows.

\end{proof}

\section{Appendix C: Continuity of exponents} 
\label{sec:appendixc}

In this part we will comment on the continuity of the exponents of $L^n\circ f$ with respect to the volume preserving diffeomorphism $f$. 

Let us assume first that $L:\mathbb T^3\rightarrow\mathbb T^3$ has the eigenvalues $\lambda_s<1<\lambda_{wu}<\lambda_{su}$, so $L^n$ has the eigenvalues $\lambda_s^n<1<\lambda_{wu}^n<\lambda_{su}^n$. Then $L^n\circ f$ will have a stable exponent $\chi^s(L^n\circ f)<0$ and two unstable exponents $0<\chi^{wu}(L^n\circ f)\leq\chi^{su}(L^n\circ f)$.

Let us first remark that for $f$ not too far from the identity, the decomposition into three invariant bundles persists, so clearly the Lyapunov exponents are continuous. In fact the stable exponent $\chi^s(L^n\circ f)$ stays always continous. However when $f$ increases, the dominated decomposition inside $E^u$ will break, and the continuity of the exponents $\chi^{wu}(L^n\circ f)$ and $\chi^{su}(L^n\circ f)$ is not obvious anymore.

We want to say something about the continuity of the map $f\mapsto \chi^{wu}(L^n\circ f)$. Going back to the initial question on the realization of Lyapunov exponents, this would imply that with our perturbations we do not realize only some set of isolated values, but we realize in fact a continuous set, and most probably with nonempty interior (with a bit of extra effort).

The continuity will not hold in the $\mathcal{C}^1$ topology because of the results of Ma\~n\'e- Bochi-Viana (above cited), thus we will work in the space of $\mathcal{C}^2$ volume preserving diffeomorphisms equipped with the $\mathcal{C}^2$ topology.

We cannot give a general affirmative answer to the question of the continuity of the exponents, however we can say some things under additional assumptions.

Backes-Brown-Butler obtained the following result.

\begin{proposition} (Corollary 3.4 in \cite{BBB}) 
Let $g : M \rightarrow M$ be a transitive $\mathcal{C}^{1+\alpha}$ Anosov diffeomorphism for some $\alpha > 0$ and $\phi: M \rightarrow\mathbb R$ a H\" older continuous potential. If $\dim E^u = 2$ and $Dg|_{E^u}$ is fiber-bunched then $g$ is a continuity point for the Lyapunov exponents $\lambda_{\pm}(Dg |_{E^{u,g}}, \mu_{\phi})$ as a function of $g\in Diff^{1+\alpha}(M)$ and $\phi\in C^{\beta}(M,\mathbb R)$.
\end{proposition}

Here $\mu_{\phi}$ is the unique equilibrium state of $g$ for $\phi$, and fiber-bunched means that there exists $n>0$ such that
\[
\|Dg_x^n|_{E^u_x}\|\cdot\|(Dg_x^n|_{E_x^u})^{-1}\|\cdot\max\{\|Dg_x^n|_{E^s_x}\|^{\beta},\|(Dg_x^n|_{E^u_x})^{-1}\|^{\beta}\}<1.
\]

In our case $g$ is the map $L^n\circ f$, $\alpha$ and $\beta$ are equal to 1, and $\phi$ is the logarithm of the unstable Jacobian of $g$, so $\mu_{\phi}$ is the volume for every $g$. Using the fact that our maps are volume preserving we obtain the following result.

\begin{proposition}\label{prop:cont}
If
\[
\|L^nDf_x|_{E^u_x}\|\cdot\|(L^nDf_x|_{E_x^u})^{-1}\|^2<1,
\]
then the map $f\mapsto \chi^{wu}(L^n\circ f)$ is continuous at $f$ in the $C^2$ topology.
\end{proposition}

Considering the case of the family $f_t$ from Section \ref{sec:prelim} we obtain the following.

\begin{proposition}
Suppose that the map $L$ satisfies the condition $\lambda_{wu}^2>\lambda_{su}$. Then there exists a constant $p_L>0$ such that the map $t\mapsto \chi^{wu}(L^n\circ f_t)$ is continuous for
\[
0\leq t<p_L\left(\frac{\lambda_{wu}^2}{\lambda_{su}}\right)^{\frac n3}.
\]
then 
\end{proposition}

\begin{proof}
Apply the Proposition \ref{prop:cont} and the following bounds that can be easily obtained from Section \ref{sec:prelim}:
\begin{eqnarray*}
\|L^nDf_t|_{E^u}\|&<&C_Lt\lambda_{su}^n,\\
\|(L^nDf_t|_{E^u})^{-1}\|&<&C_Lt\lambda_{wu}^{-n}.
\end{eqnarray*}
\end{proof}

Let us comment that this allows us to modify the Lyapunov exponents continuously along the path $t\mapsto L^n\circ f_t$, at least up to $t=p_L\left(\frac{\lambda_{wu}^2}{\lambda_{su}}\right)^{\frac n3}$ (this would increase the highest exponent for $n$ large enough). We do expect that this bound on $t$ is not necesary, in fact we have the following conjecture:

\begin{conjecture}
The Lyapunov exponents vary continuously in the space of $\mathcal{C}^2$ volume preserving Anosov diffeomorphisms on $\mathbb T^3$.
\end{conjecture}

\section{Appendix D: The case of equal eigenvalues} 
\label{sec:appendixd}

We considered in \hyperlink{theoremA}{Theorem A} the case where the two unstable eigenvalues of the linear Anosov automorphism are real and different. Here we want to say a few words about the case when the two unstable eigenvalues are equal (or there is a complex eigenvalue). In particular, is it possible to increase the top Lyapunov exponents for such an Anosov diffeomorphism in the same homotopy class?

In our method we used the fact that the eigenvalues are different in order to have a good cone with strong expansion and invariance (in the good region). It may be possible to try a similar approach, but it seems that the perturbation has to be more complicated, for example the composition of two maps similar to $f_t$ with strong shear in two different directions.

However in this case there is another simpler method to increase the top exponent. In fact one expects that a generic small perturbation will separate the two unstable eigenvalues, and if the sum is constant, the top exponent has to increase. However this alternative method does not give any estimates on the size of the Lyapunov exponents.

We will consider the case of the top eigenvalue real and with multiplicity two. The case of a complex top eigenvalue, or of two eigenvalues with the same absolute value is similar. 

\begin{proposition}
Let $L$ be an Anosov automorphism on $\mathbb T^d$ with the top unstable eigenvalue $\lambda_{su}$ real and of multiplicity two. There exists a $C^{\infty}$ diffeomorphism $f$ of $\mathbb T^d$, volume preserving and homotopic to identity, such that the top Lyapunov exponent of $L\circ f$ is strictly greater than $\log|\lambda_{su}|$.
\end{proposition}

\begin{proof}
Using standard local perturbation techniques one can show that there exists a $C^{\infty}$ diffeomorphism $f$ of $\mathbb T^d$, $C^1$ close to the identity, such that:
\begin{enumerate}
\item
$f$ preserves the planes $E^{su}$ of the strong unstable foliation of $L$, and also preserves the area on these planes (thus it is volume preserving);
\item
there exists a periodic point $p\in \mathbb T^d$ with period $n_p$ such that $L\cdot Df^{n_p}(p)|_{E^{su}}$ has complex eigenvalues which are not real multiples of a root of unity;
\item
there exists a periodic point $q\in \mathbb T^d$ with period $n_q$ such that $L\cdot Df^{n_q}(q)|_{E^{su}}$ has two different eigenvalues
\item
$L\circ f$ is Anosov.
\item
$\|Df\|,\|Df^{-1}\|$ are arbitrarily close to $1$.
\end{enumerate}

Consider $\mathbb{P}:\mathbb T^d\times \mathbb {PR}^2\rightarrow\mathbb T^d\times \mathbb {PR}^2$ to be the projectivization of the cocycle $D(L\circ f)|_{E^{su}}$:
\[
\mathbb{P}(x,[v])=(L\circ f(x),[L\cdot Df(x)v]).
\]
Since $L$ has equal eigenvalues on $E^{su}$ and $f$ is $\mathcal{C}^1$ close to the identity we have that $\mathbb{P}$ is a partially hyperbolic skew product over the hyperbolic map $L\circ f$ with one dimensional fibers.

Assume by contradiction that the top exponent of $L\circ f$ is less than equal to $\log|\lambda_{su}|$. Since the sum of the two top exponent is not changed by the perturbation, this implies that the two top Lyapunov exponent must be both equal to $\log|\lambda_{su}|$. It is well known that this implies in turn that the central Lyapunov exponent of $\mathbb{P}$ vanishes for every invariant measure $\mu$ of $\mathbb{P}$ which projects to the volume on $\mathbb T^d$. The invariance principle of Avila-Viana (\cite{AV2010}) implies that the disintegrations of $\mu$ along the fiber are continuous, invariant under $\mathbb{P}$, and invariant under the stable and unstable holonomies of $\mathbb{P}$. However from our construction $\mathbb{P}^{n_p}$ is a rotation on the fiber over $p$, so the disintegration has full support on this fiber, while $\mathbb{P}^{n_q}$ is a north-south pole map on the fiber over $q$, so the disintegration is atomic. This is a contradiction.
\end{proof}

\section{Appendix E: Physical measures} 
\label{sec:appendixE}

For an endomorphism $f:M\to M$, let us recall that an $f$-invariant measure $\mu$ is said to be physical if its basin
\[
B(\mu):=\{x\in M:\forall \phi\in\mathcal{C}^0(M), \frac{\phi(x)+\phi(fx)+\cdots \phi(f^{k-1}x)}{n}\xrightarrow[n\to\oo]{}\int \phi d\mu\}
\]
has positive Lebesgue measure. Although the existence and properties of such measures is a central topic in smooth ergodic, not too many types of examples of systems having physical measures are known, particularly if one further requires these to be abundant, i.e.\@ stable under some particular perturbations, or even part of continuous families of finitely many parameters. Instead of giving a list of known examples we refer the reader to the introduction of \cite{NUHD} where this point is discussed.

\begin{proposition}
The examples constructed in \hyperlink{theoremB}{Theorem B} posses non-uniformly hyperbolic ergodic physical measures.
\end{proposition}

Non-uniformly hyperbolic means that all Lypunov exponents are different from zero.

\begin{proof}
The result follows from the work of Pesin \cite{LyaPesin}. Since the Lyapunov exponents are non-zero almost everywhere, the volume splitts into at most countably many ergodic components and one has a decomposition $M=M_0\cup \bigcup_{n=1}^{N}M_n, N\in\mathbb{Z}\cup{+\oo}$ where $\mu(M_0)=0$ and for every $n\geq 1$ there exists an ergodic measure $\mu_n$ supported in $M_n$ ($\mu_n(M\setminus M_n)=0$). It holds that $\mu$ is can be written as $\sum_{n=1}^{\oo}a_n\mu_n$ for some non-negative sequence $(a_n)_n$, and since necessarily $\mu(M_n)>0$ for some $n$, $\mu_n$ is non-uniformly (ergodic) physical measure for $f$.
\end{proof}

It is inherent to the construction that in the strongly partially hyperbolic examples, the center bundle does not admit a dominatted splitting into one dimensional sub-bundles. Together with the remarks of \hyperlink{sec:appendixb}{Appendix B} we obtain $\mathcal{C}^2$ robust examples of conservative systems having non-uniformly hyperbolic physical measures satisfying 
\begin{itemize}
	\item there is no domination between the stable and unstable Pesin bundles (cf. \cite{LyaPesin});
	\item they in the same homotopy class of Anosov diffeomorphisms, but are not Anosov.
\end{itemize}
As far as we know no other examples of this kind are available in the literature.

One is led to wonder wether the physical measure referred above is unique; this amounts to showing that the volume is ergodic for the system. As mentioned in page \pageref{ergodicidad}, establishing this property remains for future research, but it is possible that some more general principle could be applied.

\smallskip 
 
\noindent\textbf{Question:} Suppose that $f:\Tor^4\to\Tor^4$ is a conservative (strongly) partially hyperbolic diffeomorphism with two dimensional center bundle that is non-uniformly hyperbolic. Asume further that $f$ is homotopic to an Anosov automorphism. Is $(f,\mu)$ ergodic?

\bibliographystyle{alpha}
\bibliography{bibliococyclesanosov}	

\begin{thebibliography}{BKH21}

\bibitem[AV10]{AV2010}
Artur Avila and Marcelo Viana.
\newblock Extremal {L}yapunov exponents: an invariance principle and
  applications.
\newblock {\em Invent. Math.}, 181(1):115--189, 2010.

\bibitem[AV20]{AV2020}
Artur Avila and Marcelo Viana.
\newblock Stable accessibility with 2-dimensional center.
\newblock {\em Ast\'{e}risque}, (416, Quelques aspects de la th\'{e}orie des
  syst\`emes dynamiques: un hommage \`a Jean-Christophe Yoccoz.II):301--320,
  2020.

\bibitem[Bar13]{Barreira2013}
Luis Barreira.
\newblock {\em Introduction to smooth ergodic theory}.
\newblock American Mathematical Society, Providence, Rhode Island, 2013.

\bibitem[BB03]{BARAVIERA2003}
Alexandre~T. Baraviera and Christian Bonatti.
\newblock Removing zero {L}yapunov exponents.
\newblock {\em Ergodic Theory and Dynamical Systems}, 23(6):1655--1670, dec
  2003.

\bibitem[BBB18]{BBB}
Lucas Backes, Aaron Brown, and Clark Butler.
\newblock Continuity of {L}yapunov exponents for cocycles with invariant
  holonomies.
\newblock {\em J. Mod. Dyn.}, 12:223--260, 2018.

\bibitem[BBI09]{DynCoh3Torus}
M.~Brin, D.~Burago, and S.~Ivanov.
\newblock {Dynamical Coherence of Partially Hyperbolic Diffeomorphisms of the
  3-Torus}.
\newblock {\em Journal of Modern Dynamics}, 3:1--11, 2009.

\bibitem[BC14]{NUHD}
P.~Berger and P.~D. Carrasco.
\newblock Non-uniformly hyperbolic diffeomorphisms derived from the standard
  map.
\newblock {\em Communications in Mathematical Physics}, 329(1):239--262, 2014.

\bibitem[BE20]{BE2020}
Thomas Barthelm\'{e} and Alena Erchenko.
\newblock Flexibility of geometric and dynamical data in fixed conformal
  classes.
\newblock {\em Indiana Univ. Math. J.}, 69(2):517--544, 2020.

\bibitem[BKH21]{Bochi2019}
J.~Bochi, A.~Katok, and F.~Rodriguez Hertz.
\newblock Flexibility of {L}yapunov exponents.
\newblock {\em to appear in Ergodic Theory and Dynamical Systems (Katok
  memorial issue)}, 2021.

\bibitem[BV05]{Bochi2005}
J.~Bochi and M.~Viana.
\newblock The {L}yapunov exponents of generic volume-preserving and symplectic
  maps.
\newblock {\em Annals of Mathematics}, 161(3):1423--1485, 2005.

\bibitem[BW10]{ErgPH}
K.~Burns and A.~Wilkinson.
\newblock {On the Ergodicity of Partially Hyperbolic Systems}.
\newblock {\em Ann. of Math}, 171:451--489, 2010.

\bibitem[Car20]{RandomProdSt}
Pablo~D. Carrasco.
\newblock Random products of standard maps.
\newblock {\em Communications of Mathematical Physics}, 377(2):773--810, July
  2020.

\bibitem[EK19]{EK2019}
Alena Erchenko and Anatole Katok.
\newblock Flexibility of entropies for surfaces of negative curvature.
\newblock {\em Israel J. Math.}, 232(2):631--676, 2019.

\bibitem[Erc19]{E2019}
Alena Erchenko.
\newblock Flexibility of {L}yapunov exponents for expanding circle maps.
\newblock {\em Discrete Contin. Dyn. Syst.}, 39(5):2325--2342, 2019.

\bibitem[FK60]{FurstKest60}
H.~Furstenberg and H.~Kesten.
\newblock Products of random matrices.
\newblock {\em Ann. Math. Statist.}, 31:457--469, 1960.

\bibitem[FPS14]{fisherpotriesamba}
T.~Fisher, R.~Potrie, and M.~Sambarino.
\newblock Dynamical coherence of partially hyperbolic diffeomorphisms of tori
  isotopic to {A}nosov.
\newblock {\em Mathematische Zeitschrift}, 278(1-2):149--168, 2014.

\bibitem[Fra69]{Franks1969}
J.~Franks.
\newblock {Anosov Diffeomorphisms on Tori}.
\newblock {\em Trans. of the American Math. Soc.}, pages 117--124, 1969.

\bibitem[HPS77]{HPS}
M.~Hirsch, C.~Pugh, and M.~Shub.
\newblock {\em {Invariant Manifolds}}, volume 583 of {\em {Lect. Notes in
  Math.}}
\newblock Springer Verlag, 1977.

\bibitem[HS17]{HS}
Vanderlei Horita and Martin Sambarino.
\newblock Stable ergodicity and accessibility for certain partially hyperbolic
  diffeomorphisms with bidimensional center leaves.
\newblock {\em Comment. Math. Helv.}, 92(3):467--512, 2017.

\bibitem[Man74]{Manning1974}
Anthony Manning.
\newblock There are no new {A}nosov diffeomorphisms on tori.
\newblock {\em American Journal of Mathematics}, 96(3):422, 1974.

\bibitem[Oba20]{Obata}
Davi Obata.
\newblock On the stable ergodicity of {B}erger-{C}arrasco's example.
\newblock {\em Ergodic Theory Dynam. Systems}, 40(4):1008--1056, 2020.

\bibitem[Ose68]{Oseledets}
{V.I.} Oseledets.
\newblock {A multiplicative ergodic theorem. Lyapunov characteristic numbers
  for dynamical systems.}
\newblock {\em {Trans. Mosc. Math. Soc.}}, 19:197--231, 1968.

\bibitem[Pes77]{LyaPesin}
Y.~Pesin.
\newblock {Characteristic Lyapunov exponents, and smooth ergodic theory}.
\newblock {\em Russian Math. Surveys}, 32(4):55--114, 1977.

\bibitem[PS72]{ErgAnAc}
C.~Pugh and M.~Shub.
\newblock {Ergodicity of Anosov Actions}.
\newblock {\em Inven. Math.}, 15:1--23, 1972.

\bibitem[PT14]{PT}
G.~Ponce and A.~Tahzibi.
\newblock Central {L}yapunov exponent of partially hyperbolic diffeomorphisms
  of {$\Bbb{T}^3$}.
\newblock {\em Proc. Amer. Math. Soc.}, 142(9):3193--3205, 2014.

\bibitem[Rok62]{Rokhlin}
V.~Rokhlin.
\newblock {On the Fundamental Ideas of Measure Theory}.
\newblock {\em Transl. Amer. Math. Soc.}, 10:1--52, 1962.

\bibitem[Sin68]{SinaiMarkov}
Ya. Sinai.
\newblock Markov partitions and {C}-diffeomorphisms.
\newblock {\em Functional Analysis and Its Applications}, 2(1):61--82, 1968.

\bibitem[SW00]{Patho}
M.~Shub and A.~Wilkinson.
\newblock {Pathological Foliations and Removable Zero Exponents}.
\newblock {\em Inv. Math.}, 139(139):495--508, 2000.

\end{thebibliography}
\end{document}